\documentclass[reqno,10pt]{amsart}
\usepackage{amscd,amssymb,verbatim,array}
\usepackage{hyperref}
\usepackage{amsmath}
\usepackage[active]{srcltx} 

\numberwithin{equation}{section} \theoremstyle{plain}

\newtheorem{theorem}{Theorem}[section]
\newtheorem{lemma}[theorem]{Lemma}

\theoremstyle{definition}

\theoremstyle{remark}

\numberwithin{equation}{section}

\newcommand{\Det}{\operatorname{Det}}
\newcommand{\B}{\mathcal{B}}

\newcommand{\PP}{\mathcal{P}}

\newcommand{\graph}{\operatorname{graph}}
\newcommand{\Fr}{\operatorname{Fr}}

\newcommand{\Tan}{\operatorname{tan}}
\newcommand{\Nor}{\operatorname{nor}}
\newcommand{\Dim}{\operatorname{dim}}

\newcommand{\rel}{\operatorname{rel}}
\newcommand{\Abs}{\operatorname{abs}}

\newcommand{\pr}{\operatorname{pr}}

\newcommand{\Ker}{\operatorname{ker}}

\newcommand{\Span}{\operatorname{span}}

\newcommand{\Imm}{\operatorname{Im}}
\newcommand{\Dom}{\operatorname{Dom}}

\newcommand{\ddet}{\operatorname{det}}
\newcommand{\Id}{\operatorname{Id}}

\begin{document}

\title[The gluing formula of the zeta-determinants of Dirac Laplacians]
{The gluing formula of the zeta-determinants of Dirac Laplacians for certain boundary conditions}

\author{Rung-Tzung Huang}

\address{Department of Mathematics, National Central University, Chung-Li 320, Taiwan, Republic of China}

\email{rthuang@math.ncu.edu.tw}

\author{Yoonweon Lee}

\address{Department of Mathematics, Inha University, Incheon, 402-751, Korea}

\email{yoonweon@inha.ac.kr}

\thanks{The first author was supported by the National Science Council, Republic of China with the grant number NSC 102-2115-M-008-005 and
the second author was partially supported by the National Research Foundation of Korea with the grant number NRF-2012R1A1A2001086.}

\begin{abstract}
The odd signature operator is a Dirac operator which acts on the space of differential forms of all degrees and whose square is the usual Laplacian.
We extend the result of [15] to prove the gluing formula of the zeta-determinants of Laplacians acting on differential forms of all degrees
with respect to the boundary conditions ${\mathcal P}_{-, {\mathcal L}_{0}}$, ${\mathcal P}_{+, {\mathcal L}_{1}}$.
We next consider a double of de Rham complexes consisting of differential forms of all degrees with the absolute and relative boundary conditions.
Using a similar method, we prove the gluing formula of the zeta-determinants of Laplacians acting on differential forms of all degrees
with respect to the absolute and relative boundary conditions.
\end{abstract}
\maketitle

\section{Introduction}

\vspace{0.2 cm}

The zeta-determinants of Laplacians are global spectral invariants on compact Riemannian manifolds with or without boundary,
which play central roles in the theory of the analytic torsions and other related fields.
For a global invariant the gluing formula is very useful in various kinds of computations.
The gluing formula of the zeta-determinants of Laplacians was proved by
D. Burghelea, L. Friedlander and T. Kappeler in [5] by using the Dirichlet boundary condition and the Dirichlet-to-Neumann operator,
which we call the BFK-gluing formula.
Because of relations to topology the relative and absolute boundary conditions are commonly used for Hodge Laplacians.
However, the gluing formula for the zeta-determinants of Hodge Laplacians with respect to these boundary conditions
is not known yet. In this paper we discuss this problem in a weak sense. More precisely, we
prove the gluing formula for the zeta-determinants of Hodge Laplacians
acting on the space of differential forms of all degrees, not a single space of $q$-forms, with respect to the relative and absolute boundary conditions
(Theorem \ref{Theorem:4.2}).

K. Wojciechowski and S. Scott studied the zeta-determinants of Dirac Laplacians on compact Riemannian manifolds with boundary,
acting on Clifford module bundles with respect to boundary conditions belonging to the smooth self-adjoint Grassmannian
including the Atiyah-Patodi-Singer (APS) boundary condition
and the Calder\'on projector ([18], [19], [20], [26]).
Using their results and the BFK-gluing formula, P. Loya, J. Park ([16], [17]) and the second author ([15]) studied independently
the gluing formula of Dirac Laplacians
with respect to boundary conditions belonging to the smooth self-adjoint Grassmannian on compact Riemannian manifolds.

M. Braverman and T. Kappeler studied the refined analytic torsion on a closed odd dimensional Riemannian manifold
by using the odd signature operator ([3], [4]),
as an analytic analogue of the refined combinatorial torsion developed by M. Farber and V. Turaev ([6], [7], [22], [23]).
The boundary problem of the refined analytic torsion was studied by B. Vertman ([24], [25]) and the authors ([9], [10], [11]) in different ways.
Vertman used a double of de Rham complexes consisting of differential forms satisfying the absolute and relative boundary conditions.
The authors introduced well-posed boundary conditions
${\mathcal P}_{-, {\mathcal L}_{0}}$, ${\mathcal P}_{+, {\mathcal L}_{1}}$ for the odd signature operator to define the refined analytic torsion on compact Riemannian manifolds with boundary.
In [11] the authors compared these two constructions.

We note that the odd signature operator is a Dirac operator which acts on the space of differential forms of all degrees and whose square is the usual Laplacian.
In this paper we extend the result of [15] to other class of boundary conditions and discuss the gluing formula of the zeta-determinants of Laplacians
acting on the space of differential forms of all degrees with respect to
${\mathcal P}_{-, {\mathcal L}_{0}}/{\mathcal P}_{+, {\mathcal L}_{1}}$ (Theorem \ref{Theorem:1.1}) and the absolute/relative boundary conditions (Theorem \ref{Theorem:4.2}).
In case of the absolute/relative boundary conditions we are going to use the double of De Rham complexes which was used by B. Vertman in [24].

\vspace{0.3 cm}

\section{Review of the gluing formula of the zeta-determinants of Dirac Laplacians}

\vspace{0.2 cm}

In this section we review and extend the results in [15].
Let $(M, g)$ be an $m$-dimensional compact oriented Riemannian manifold with boundary $Y$ and $E \rightarrow M$ be a Hermitian vector bundle.
Choose a collar neighborhood $N$ of $Y$ which is diffeomorphic to $[0, 1) \times Y$. We assume that the metric $g$ is a product one on $N$
and the bundle $E$ has the product structure on $N$, which means that $E|_{N} = p^{\ast} (E|_{Y})$, where $p : [0, 1) \times Y \rightarrow Y$
is the canonical projection. Let ${\mathcal D}_{M}$ be a Dirac type operator acting on smooth sections of $E$ and
satisfying the following conditions : (1) On the collar neighborhood $N$ of $Y$ ${\mathcal D}_{M}$ has the following form

\begin{eqnarray}  \label{E:2.1}
{\mathcal D}_{M} & = & G \hspace{0.1 cm} \left( \partial_{u} \hspace{0.1 cm} + \hspace{0.1 cm} {\mathcal A} \right),
\end{eqnarray}

\vspace{0.2 cm}

\noindent
where $G : E|_{Y} \rightarrow E|_{Y}$ is a bundle automorphism with $G^{2} = - \Id$, $\partial_{u}$ is the inward normal derivative to $Y$ and ${\mathcal A}$ is the tangential Dirac operator.
(2) $G$ and ${\mathcal A}$ are independent of the normal coordinate $u$ and satisfy

\vspace{0.2 cm}

\begin{eqnarray}  \label{E:2.2}
& & G^{\ast} = - G, \qquad G^{2} = - \Id, \qquad {\mathcal A}^{\ast} = {\mathcal A}, \qquad G {\mathcal A} = - {\mathcal A}G    \nonumber \\
& & \Dim ( \Ker ( G - i ) \cap \Ker {\mathcal A} ) = \Dim ( \Ker ( G + i ) \cap \Ker {\mathcal A} ).
\end{eqnarray}

\vspace{0.2 cm}

\noindent
Then, on $N$, the Dirac Laplacian ${\mathcal D}_{M}^{2}$ has the following form

\begin{eqnarray}  \label{E:2.3}
{\mathcal D}_{M}^{2} & = & - \partial_{u}^{2} \hspace{0.1 cm} + \hspace{0.1 cm} {\mathcal A}^{2}.
\end{eqnarray}

\vspace{0.2 cm}

We next introduce boundary conditions on $Y$.
The Dirichlet boundary condition on $Y$ is defined by the restriction map $\gamma_{0} : C^{\infty}(M) \rightarrow C^{\infty}(Y)$,
$\gamma_{0}(\phi) = \phi|_{Y}$ and the realization ${\mathcal D}_{M, \gamma_{0}}^{2}$ is defined to be the operator ${\mathcal D}_{M}^{2}$
with the following domain

\begin{eqnarray}  \label{E:2.4}
\Dom \left( {\mathcal D}_{M, \gamma_{0}}^{2} \right) & = & \{ \phi \in C^{\infty}(M) \mid \phi|_{Y} = 0 \}.
\end{eqnarray}

\noindent
Then ${\mathcal D}_{M, \gamma_{0}}^{2}$ is an invertible operator by the unique continuation property of ${\mathcal D}_{M}$ ([12], [1]).

The APS boundary condition $\Pi_{>}$ (or $\Pi_{<}$) is defined to be the orthogonal projection onto the space spanned by the positive (or negative)
eigensections of ${\mathcal A}$. If $\Ker {\mathcal A} \neq \{ 0 \}$, $\Ker {\mathcal A}$ is an even dimensional vector space by (\ref{E:2.2}). We choose a unitary operator
$\sigma : \Ker {\mathcal A} \rightarrow \Ker {\mathcal A}$ satisfying

\begin{eqnarray}  \label{E:2.5}
\sigma G  & = & - G \sigma, \qquad \sigma^{2} \hspace{0.1 cm} = \hspace{0.1 cm} \Id_{\Ker {\mathcal A}}.
\end{eqnarray}

\vspace{0.2 cm}

\noindent
We put $\sigma^{\pm} := \frac{ I \pm \sigma}{2}$ and define $\Pi_{<, \sigma^{-}}$, $\Pi_{>, \sigma^{+}}$ by

\begin{eqnarray}  \label{E:2.6}
\Pi_{<, \sigma^{-}} & := & \Pi_{<} + \frac{1}{2} ( I - \sigma)|_{\Ker {\mathcal A}}, \qquad
\Pi_{>, \sigma^{+}} \hspace{0.1 cm} := \hspace{0.1 cm} \Pi_{>} + \frac{1}{2} ( I + \sigma)|_{\Ker {\mathcal A}}.
\end{eqnarray}

\vspace{0.2 cm}

\noindent
The realizations ${\mathcal D}_{M, \Pi_{<, \sigma^{-}}}$ and ${\mathcal D}^{2}_{M, \Pi_{<, \sigma^{-}}}$ are defined to be
${\mathcal D}_{M}$ and ${\mathcal D}^{2}_{M}$ with the following domains.

\begin{eqnarray}  \label{E:2.7}
\Dom \left( {\mathcal D}_{M, \Pi_{<, \sigma^{-}}} \right) & = & \{ \phi \in C^{\infty}(M) \mid \Pi_{<, \sigma^{-}} (\phi|_{Y}) = 0 \} \nonumber \\
\Dom \left( {\mathcal D}^{2}_{M, \Pi_{<, \sigma^{-}}} \right) & = & \{ \phi \in C^{\infty}(M) \mid \Pi_{<, \sigma^{-}} (\phi|_{Y}) = 0,
\quad  \Pi_{<, \sigma^{-}} (( {\mathcal D}_{M} \phi)|_{Y}) = 0 \}.
\end{eqnarray}

\vspace{0.2 cm}

\noindent
${\mathcal D}_{M, \Pi_{>, \sigma^{+}}}$ and ${\mathcal D}^{2}_{M, \Pi_{>, \sigma^{+}}}$ are defined similarly.
The Calder\'on projector ${\mathcal C}$ is defined to be the orthogonal projection from $L^{2}(E|_{Y})$ onto the closure of
$\{ \phi|_{Y} \mid \phi \in C^{\infty}(M),  {\mathcal D}_{M} \phi = 0\}$ called the Cauchy data space.

\vspace{0.2 cm}

As a generealization of the APS boundary condition, K. Wojciekowski and B. Booss introduced the smooth self-adjoint Grassmannian
$Gr^{\ast}_{\infty}({\mathcal D}_{M})$ ([2], [20], [26]),
which is the set of all orthogonal pseudodifferential projections $P$ such that

\begin{eqnarray}    \label{E:2.8}
-GPG  = I - P, \qquad P - \Pi_{>} \hspace{0.2 cm} \text{is a classical pseudodifferential operator of order} \hspace{0.2 cm} - \infty.
\end{eqnarray}

\vspace{0.2 cm}

\noindent
Clearly, $\Pi_{>, \sigma^{+}}$ belongs to $Gr^{\ast}_{\infty}({\mathcal D}_{M})$. It was known by S. Scott ([18]) and G. Grubb ([8]) that ${\mathcal C}$
belongs to $Gr^{\ast}_{\infty}({\mathcal D}_{M})$.
The realizations ${\mathcal D}_{M, P}$ and ${\mathcal D}^{2}_{M, P}$ are similarly defined as (\ref{E:2.7}) by simply replacing $\Pi_{<, \sigma^{-}}$
with $P$.

\vspace{0.2 cm}

Since $G$ is a bundle automorphism on $E|_{Y}$ with $G^{2} = - \Id$, $E|_{Y}$ splits onto $\pm i$-eigenspaces $E_{Y}^{\pm}$, say,
$E|_{Y} = E_{Y}^{+} \oplus E_{Y}^{-}$ and the Dirac operator ${\mathcal D}_{M}$ can be written, near the boundary $Y$, by

\begin{eqnarray}   \label{E:2.9}
{\mathcal D}_{M} & = & \left( \begin{array}{clcr} i & 0 \\ 0 & - i \end{array} \right) \left( \partial_{u} +
\left( \begin{array}{clcr} 0 & {\mathcal A}^{-} \\ {\mathcal A}^{+} & 0 \end{array} \right) \right),
\end{eqnarray}

\noindent
where ${\mathcal A}^{\pm} := {\mathcal A}|_{C^{\infty}(E_{Y}^{\pm})} :  C^{\infty}(E_{Y}^{\pm}) \rightarrow C^{\infty}(E_{Y}^{\mp})$ and
$({\mathcal A}^{\pm})^{\ast} = {\mathcal A}^{\mp}$.
For $P \in Gr^{\ast}_{\infty}({\mathcal D}_{M})$, there exists a unitary operator $U_{P} : L^{2}(E_{Y}^{+}) \rightarrow L^{2}(E_{Y}^{-})$
such that $\graph (U_{P}) = \Imm P$. For simplicity we write $U_{{\mathcal C}} = K$.
By (\ref{E:2.8}) we have

\begin{eqnarray}  \label{E:2.10}
U_{P} & = & K \hspace{0.1 cm} + \hspace{0.1 cm} \text{a} \hspace{0.1 cm} \text{smoothing} \hspace{0.1 cm}  \text{operator}.
\end{eqnarray}

\vspace{0.2 cm}

We introduce the Neumann jump operator $Q(t) : C^{\infty}(Y) \rightarrow C^{\infty}(Y)$ for $t \geq 0$ as follows.
For $f \in C^{\infty}(Y)$ there exists a unique section $\phi \in C^{\infty}(E)$ satisfying $({\mathcal D}^{2}_{M} + t ) \phi = 0$, $\phi|_{Y} = f$.
Then we define

\begin{eqnarray}   \label{E:2.11}
Q(t)(f) & = & - (\partial_{u} \phi)|_{Y}.
\end{eqnarray}

\vspace{0.2 cm}

\noindent
The Green formula shows that $Q(t) - {\mathcal A}$ is a non-negative operator and $\Ker (Q - {\mathcal A}) = \Imm {\mathcal C}$, the Cauchy data space
(Lemma 2.5 in [15]), where $Q:= Q(0)$.
Moreover, $Q - | {\mathcal A} |$ (Theorem 2.1 in [14]) and $P - \Pi_{>}$ are smoothing operators,
which implies that $(I - P) (Q - {\mathcal A}) (I - P)$ differs from $2 \Pi_{<} | {\mathcal A} |$
by a smoothing operators. Hence the zeta determinant of $(I - P) (Q - {\mathcal A}) (I - P)$ is well defined even though
$(I - P) (Q - {\mathcal A}) (I - P)$ is not an elliptic operator.
It is not difficult to show that $\Ker (I - P) (Q - {\mathcal A}) (I - P) =\{ \psi|_{Y} \mid \psi \in \Ker {\mathcal D}_{M, P} \}$ (Lemma 2.5 in [15]).
Let $\{ h_{1}, \cdots, h_{q} \}$ be an orthonormal basis for $\Ker (I - P) (Q - {\mathcal A}) (I - P)$, where $ q = \Dim \Ker {\mathcal D}^{2}_{M, P}$.
Then there exist $\psi_{1}, \cdots, \psi_{q} \in \Ker {\mathcal D}^{2}_{M, P}$ with $\psi_{i}|_{Y} = h_{i}$.
We define a $q \times q$ positive definite Hermitian matrix $V_{M, P}$ by

\begin{eqnarray}    \label{E:2.12}
V_{M, P} := (v_{ij}), \qquad v_{ij} = \langle \psi_{i}, \psi_{j} \rangle_{M}.
\end{eqnarray}

\vspace{0.2 cm}

If ${\frak P}$ is an invertible elliptic operator of order $>0$ with discrete spectrum $\{ \lambda_{j} \mid j = 1, 2, 3, \cdots \}$,
we define the zeta function by $\zeta_{{\frak P}}(s) = \sum_{j=1}^{\infty} \lambda_{j}^{-s}$ and the zeta-determinant $\Det {\frak P}$ by
$e^{- \zeta_{{\frak P}}^{\prime}(0)}$.
If ${\frak P}$ has a non-trivial kernel, we define the modified zeta-determinant $\Det^{\ast} {\frak P}$ by

\begin{eqnarray}     \label{E:2. 13}
\Det^{\ast} {\frak P} := \Det \left( {\frak P} + \pr_{\Ker {\frak P}} \right).
\end{eqnarray}

\noindent
Similarly, if $\alpha$ is a trace class operator, we define the modified Fredholm determinant by

\begin{eqnarray}   \label{E:2.14}
\ddet^{\ast}_{\Fr} \left( I + \alpha \right) := \ddet \left( I + \alpha + \pr_{\Ker (I + \alpha)} \right).
\end{eqnarray}

\noindent
Equivalently, $\Det^{\ast} {\frak P}$ and $\ddet^{\ast}_{\Fr} \left( I + \alpha \right)$ are the determinants of ${\frak P}$ and $ I + \alpha $
when restricted to the orthogonal complements of $\Ker {\frak P}$ and $\Ker (I + \alpha)$, respectively.

\vspace{0.2 cm}

The following results are due to S. Scott and K. Wojciechowski ([19], [20], [26]), P. Loya and J. Park ([16], [17]) and the second author ([15]).

\begin{theorem} \label{Theorem:2.1}
Let $(M, g)$ be a compact oriented Riemannian manifold with boundary $Y$ having the product structure near $Y$.
We denote by ${\mathcal D}_{M}$
a Dirac type operator which has the form (\ref{E:2.1}) and satisfies (\ref{E:2.2}) near $Y$.
Let ${\mathcal P}$ be a pseudodifferential projection belonging to
$Gr^{\ast}_{\infty}({\mathcal D}_{M})$. Then :

\begin{eqnarray}  \label{E:2.99}
 \log \Det^{\ast} {\mathcal D}^{2}_{M, \PP}  -  \log \Det {\mathcal D}^{2}_{M, \gamma_{0}}
& = & \log \ddet V_{M, \PP} + \log \Det^{\ast} \left( (I - \PP) (Q - {\mathcal A}) (I - \PP) \right) ,  \\
\label{E:2.999}
\log \Det^{\ast} {\mathcal D}^{2}_{M, \PP}  -   \log \Det {\mathcal D}^{2}_{M, {\mathcal C}}  & = & 2 \log \ddet V_{M, \PP} +
2 \log | \ddet^{\ast}_{\Fr} \left( \frac{1}{2} \left( I + U_{\PP}^{-1} K \right) \right) | ,
\end{eqnarray}

\noindent
where $(I - \PP) (Q - {\mathcal A}) (I - \PP)$ is considered to be an operator defined on $\Imm (I - \PP)$.
\end{theorem}

\vspace{0.2 cm}

Next we extend Theorem \ref{Theorem:2.1} to a certain pseudodifferential projection ${\mathcal P}$
satisfying the following conditions.

\vspace{0.2 cm}

\noindent
{\bf Condition A}: (1) ${\mathcal P} : L^{2}(Y, E|_{Y}) \rightarrow L^{2}(Y, E|_{Y})$ is a pseudodifferential projection which gives
a well-posed boundary condition with respect to ${\mathcal D}_{M}$ in the sense of Seeley ([8], [21]).
(2) $\Imm {\mathcal P} = \graph (U_{{\mathcal P}})$, where $U_{{\mathcal P}} : L^{2}(E_{Y}^{+}) \rightarrow L^{2}(E_{Y}^{-})$ is a unitary operator.
(3) $ U_{{\mathcal P}}^{\ast} U_{\Pi_{>, \sigma^{+}}} + U_{\Pi_{>, \sigma^{+}}}^{\ast} U_{{\mathcal P}}$ is a trace class operator
and a $\Psi$DO of order at most $-1$.
(4) The zeta-determinants of $(I - \PP) (Q(t) - {\mathcal A}) (I - \PP)$ and $\PP (Q(t) - {\mathcal A}) \PP$  for $t \geq 0$ are well defined and have asymptotic expansions for $t \rightarrow \infty$ with zero constant term.

\vspace{0.2 cm}

\noindent
{\it Remark} : A pseudodifferential projection belonging to $Gr^{\ast}_{\infty}({\mathcal D}_{M})$ satisfies the items (1), (2) and (4)
but not (3) in the {\bf Condition A} above.

\vspace{0.3 cm}

\noindent
The following lemma is straightforward by $(\ref{E:2.10})$.

\begin{lemma}  \label{Lemma:2.1}
If ${\mathcal P}$ satisfies the {\bf Condition A}, then
$U_{{\mathcal P}}^{-1} K + K^{-1} U_{{\mathcal P}}$ is a trace class operator on $L^{2}(E_{Y}^{+})$.
\end{lemma}

\vspace{0.2 cm}

\noindent
The proof of the following result is a verbatim repetition of the proof of Theorem 1.1 in [15], which is an analogue of (\ref{E:2.99}).

\vspace{0.2 cm}

\begin{theorem} \label{Theorem:2.3}
Let $(M, g)$ be a compact oriented Riemannian manifold with boundary $Y$ having the product structure near $Y$.
We denote by ${\mathcal D}_{M}$
a Dirac type operator which has the form (\ref{E:2.1}) and satisfies (\ref{E:2.2}) near $Y$.
Let ${\mathcal P}$ be a well-posed boundary condition with respect to ${\mathcal D}_{M}$ satisfying the {\bf Condition A}.
Then the following equality holds.

\begin{eqnarray*}
 \log \Det^{\ast} {\mathcal D}^{2}_{M, \PP} & - & \log \Det {\mathcal D}^{2}_{M, \gamma_{0}}
\hspace{0.2 cm} = \hspace{0.2 cm} \log \ddet V_{M, \PP} + \log \Det^{\ast} \left( (I - \PP) (Q - {\mathcal A}) (I - \PP) \right)  \\
& = & \log \ddet V_{M, \PP} + \log \Det^{\ast} \left( 2 (I - \PP) (Q - {\mathcal A}) (I - \PP) \right) - \log 2 \cdot \zeta_{(I - \PP) (Q - {\mathcal A}) (I - \PP)}(0),
\end{eqnarray*}

\noindent
where $(I - \PP) (Q - {\mathcal A}) (I - \PP)$ is considered to be an operator defined on $\Imm (I - \PP)$.
\end{theorem}

\vspace{0.2  cm}

\noindent
Theorem \ref{Theorem:2.3} and (\ref{E:2.99}) in Theorem \ref{Theorem:2.1} lead to the following result, which
is an analogue of (\ref{E:2.999}).

\vspace{0.2 cm}

\begin{theorem} \label{Theorem:2.4}
We assume the same assumptions and notations as in Theorem \ref{Theorem:2.3}. Then :

\begin{eqnarray*}
\frac{\Det^{\ast} {\mathcal D}^{2}_{M, \PP}}{\Det {\mathcal D}^{2}_{M, {\mathcal C}}} = (\ddet V_{M, \PP})^{2} \cdot
\ddet^{\ast}_{\Fr} \left( I + \frac{1}{2} \left( U_{\PP}^{-1} K + K^{-1} U_{\PP} \right) \right) \cdot 2^{- \zeta_{(I - \PP) (Q - {\mathcal A}) (I - \PP)}(0)}.
\end{eqnarray*}
\end{theorem}

\begin{proof}
The proof is almost verbatim repetition of the proof of Theorem 1.2 in [15]. We here present the proof very briefly and
refer to [15] for details .
We first define $U$, $L$ by

\begin{eqnarray*}
U & = & \Imm ( I - \PP ) \cap \Imm {\mathcal C} = \Ker ( I - \PP)( Q - {\mathcal A})(I - \PP) = \{ \phi|_{Y} \mid {\mathcal D}_{M} \phi = 0, \hspace{0.1 cm} \phi|_{Y} \in \Imm (I - \PP) \}, \\
L & = & ( I - U_{\PP})^{-1}(U) = ( I + K)^{-1}(U) = \{ x \in L^{2}(E^{+}_{Y}) \mid U_{P} x = - Kx \}.
\end{eqnarray*}

\noindent
We denote by $\Imm (I - \PP)^{\ast}$ and $L^{2}(E_{Y}^{+})^{\ast}$ the orthogonal complements of $U$, $L$ so that

$$
\Imm (I - \PP) \hspace{0.1 cm} = \hspace{0.1 cm} \Imm (I - \PP)^{\ast} \oplus U, \qquad  L^{2}(E_{Y}^{+}) \hspace{0.1 cm} = \hspace{0.1 cm} L^{2}(E_{Y}^{+})^{\ast} \oplus L.
$$

\noindent
The item (3) in the {\bf Condition A} implies that $\left( I + K^{-1}U_{\PP} \right)|_{L^{2}(E^{+}_{Y})^{\ast}} : L^{2}(E_{Y}^{+})^{\ast}
\rightarrow L^{2}(E_{Y}^{+})^{\ast}$ is an invertible operator.
For simplicity, we write $\left( \left( I + K^{-1}U_{\PP} \right)|_{L^{2}(E^{+}_{Y})^{\ast}} \right)^{-1}$ by $\left( I + K^{-1} U_{\PP} \right)^{-1}$.
We proceed as (3.5) in the proof of Theorem 1.2 in [15]. Then :

\begin{eqnarray*}
& & \log \Det^{\ast} \left( 2(I - \PP)(Q - {\mathcal A})(I - \PP) \right)  \hspace{0.1 cm}  = \hspace{0.1 cm}
\log \Det \left( 2 (I - \PP)(Q - {\mathcal A})(I - \PP) + \pr_{U} \right) \\
& = & \log \ddet_{\Fr} \left( \frac{1}{2}  ( I + K^{-1} U_{\PP}) ( I + U_{\PP}^{-1}K ) + \pr_{L} (I - K)^{-1} (Q - {\mathcal A})^{-1} ( I - K) \pr_{L} \right) \\
& & \hspace{1.0 cm} + \hspace{0.1 cm} \log \Det \left( (I - K)^{-1} (Q - {\mathcal A}) ( I - K) \right)  \\
& = & \log \ddet_{\Fr}^{\ast} \left( I + \frac{1}{2} ( K^{-1} U_{\PP} + U_{\PP}^{-1} K ) \right) +
\log \ddet \left( \pr_{L} (I - K)^{-1} (Q - {\mathcal A})^{-1} ( I - K) \pr_{L} \right) \\
& & \hspace{1.0 cm} + \hspace{0.1 cm} \log \Det \left( (I - {\mathcal C}) (Q - {\mathcal A}) ( I - {\mathcal C}) \right).
\end{eqnarray*}

\noindent
Lemma 3.1 in [15] shows that $\ddet \left( \pr_{L} (I - K)^{-1} (Q - {\mathcal A})^{-1} ( I - K) \pr_{L} \right) = \ddet V_{M, \PP}$,
from which together with Theorem \ref{Theorem:2.3} the result follows.
\end{proof}

\noindent
{\it Remark} : The kernel of $ \left( I + \frac{1}{2} ( K^{-1} U_{\PP} + U_{\PP}^{-1} K ) \right)$ is $L$ and hence we may write

\begin{eqnarray*}
\ddet_{\Fr}^{\ast} \left( I + \frac{1}{2} ( K^{-1} U_{\PP} + U_{\PP}^{-1} K ) \right) & = &
\ddet_{\Fr} \left( I + \frac{1}{2} ( K^{-1} U_{\PP} + U_{\PP}^{-1} K ) \right)|_{L^{2}(E_{Y}^{+})^{\ast}}.
\end{eqnarray*}

We next discuss the gluing formula of the zeta-determinants of Dirac Laplacians.
Let $({\widehat M}, {\widehat g})$ be a closed Riemannian manifold and $Y$ be a hypersurface of ${\widehat M}$ such that
${\widehat M} - Y$ has two components. We denote by $M_{1}$, $M_{2}$ the closure of each component, {\it i.e.} ${\widehat M} = M_{1} \cup_{Y} M_{2}$.
We assume that ${\widehat g}$ is a product metric on
a collar neighborhood $N$ of $Y$ and $N$ is isometric to $(-1, 1) \times Y$.
Let ${\widehat E} \rightarrow {\widehat M}$ be a Hermitian vector bundle having the product structure on $N$
and ${\mathcal D}_{{\widehat M}}$ be a Dirac type operator acting on smooth sections of ${\widehat E}$
which has the form, on $N$, ${\mathcal D}_{{\widehat M}} = G ( \partial_{u} + {\mathcal A})$ and satisfies (\ref{E:2.2}) as before.
Without loss of generality we assume that $\partial_{u}$ points outward on the boundary of $M_{1}$ and points inward on the boundary of $M_{2}$.
We denote by ${\mathcal D}_{M_{1}}$, ${\mathcal D}_{M_{2}}$ the restriction of ${\mathcal D}_{{\widehat M}}$ to $M_{1}$, $M_{2}$
and denote by $\gamma_{0}$ the restriction map to
 $Y$.
Suppose that $\{ h_{1}, \cdots, h_{q} \}$ is an orthonormal basis for
$(\Ker {\mathcal D}^{2}_{{\widehat M}})|_{Y} := \{ \Phi|_{Y} \mid {\mathcal D}^{2}_{{\widehat M}} \Phi = 0 \}$, where $q = \Dim \Ker {\mathcal D}_{{\widehat M}}$.
Then there exist $\Phi_{1}, \cdots, \Phi_{q}$ in $\Ker {\mathcal D}^{2}_{{\widehat M}}$ with $\Phi_{i}|_{Y} = h_{i}$.
We define a positive definite Hermitian matrix $A_{0}$ by

\begin{eqnarray}   \label{E:2.15}
A_{0} = (a_{ij}), \qquad \text{where} \quad a_{ij} = \langle \Phi_{i}, \Phi_{j} \rangle_{{\widehat M}}.
\end{eqnarray}

\noindent
Let ${\mathcal C}_{1}$, ${\mathcal C}_{2}$ be Calder\'on projectors for ${\mathcal D}_{M_{1}}$, ${\mathcal D}_{M_{2}}$ and
$K_{1}$, $K_{2} : C^{\infty}(E_{Y}^{+}) \rightarrow C^{\infty}(E_{Y}^{-})$ be unitary operators
such that $\graph(K_{i}) = \Imm {\mathcal C}_{i}$, $i = 1, 2$.
The following result is due to P. Loya, J. Park ([16], [17]) and the second author ([15]), independently.

\begin{eqnarray}  \label{E:2.16}
& & \log \Det^{\ast} {\mathcal D}_{{\widehat M}}^{2} - \log \Det {\mathcal D}^{2}_{M_{1}, {\mathcal C}_{1}} -
\log \Det {\mathcal D}^{2}_{M_{2}, {\mathcal C}_{2}} \nonumber   \\
&  & \hspace{0.2 cm} = \hspace{0.1 cm} - \log 2 \cdot ( \zeta_{{\mathcal A}^{2}}(0) + l )  \hspace{0.1 cm} + \hspace{0.1 cm} 2 \log \ddet A_{0}
\hspace{0.1 cm} + \hspace{0.1 cm}
2 \log | \ddet^{\ast}_{\Fr} \left( \frac{1}{2} \left( I - K_{1}^{-1} K_{2} \right) \right) |,
\end{eqnarray}

\noindent
where $l = \Dim \Ker {\mathcal A}$.

\vspace{0.2 cm}

\noindent
{\it Remark} : We note that ${\mathcal D}_{{\widehat M}} = G ( \partial_{u} + {\mathcal A}) = - G ( - \partial_{u} - {\mathcal A})$ near $Y$.
We use the form $G ( \partial_{u} + {\mathcal A})$ on $M_{2}$ so that
$K_{2} = U_{\Pi_{>}} + {\frak F}_{2}$ for some smoothing operator ${\frak F}_{2}$ by (\ref{E:2.10}).
Similarly, We use the form $-G ( - \partial_{u} - {\mathcal A})$ on $M_{1}$ so that
$K_{1} = U_{\Pi_{<}} + {\frak F}_{1}$ for some smoothing operator ${\frak F}_{1}$.
Since $U_{\Pi_{<}} = - U_{\Pi_{>}}$, $K_{2} = - K_{1} + {\frak F}$ for some smoothing operator ${\frak F}$ and
hence $\frac{1}{2} \left( I - K_{1}^{-1} K_{2} \right)$ is of the form $I + \alpha$ for some trace class operator $\alpha$.
Moreover, The kernel of $\hspace{0.1 cm} I - K_{1}^{-1} K_{2} \hspace{0.1 cm}$ consists of $x \in L^{2}(E_{Y}^{+})$ such that
$x + K_{1} x \hspace{0.1 cm} ( \hspace{0.1 cm} = x + K_{2} x)$ can be extended to a
harmonic section of ${\widehat {\mathcal D}}$ on ${\widehat M}$.

\vspace{0.2 cm}

\noindent
Theorem \ref{Theorem:2.4} and (\ref{E:2.16}) lead to the following result, which is an analogue of Theorem 1.3 in [15].

\begin{theorem} \label{Theorem:2.5}
Let $\PP_{1}$, $\PP_{2}$ be orthogonal pseudodifferential projections satisfying the {\bf Condition A} with respect to $M_{1}$ and $M_{2}$,
respectively. Suppose that for $i = 1, 2$, $U_{\PP_{i}} : C^{\infty}(E_{Y}^{+}) \rightarrow C^{\infty}(E_{Y}^{-})$ is a unitary operator
such that $\graph(U_{\PP_{i}}) = \Imm \PP_{i}$. We also denote by ${\mathcal A}_{i}$ the tangential Dirac operator of ${\mathcal D}_{M_{i}}$
and by $Q_{i}$ the Neumann jump operator with respect to ${\mathcal D}^{2}_{M_{i}}$ on $M_{i}$.
Then the following equality holds.

\begin{eqnarray*}
&  & \log \Det^{\ast} {\mathcal D}_{{\widehat M}}^{2} - \log \Det^{\ast} {\mathcal D}^{2}_{M_{1}, \PP_{1}} -
\log \Det^{\ast} {\mathcal D}^{2}_{M_{2}, \PP_{2}}  \\
& = & - \log 2 \cdot ( \zeta_{{\mathcal A}^{2}}(0) + l ) \hspace{0.1 cm} + \hspace{0.1 cm} 2 \log \ddet A_{0}
 - 2 \sum_{i=1}^{2}\log \ddet V_{M_{i, \PP_{i}}} + 2 \log | \ddet^{\ast}_{\Fr} \left( \frac{1}{2} \left( I - K_{1}^{-1} K_{2} \right) \right) | \\
& & - \hspace{0.1 cm} \sum_{i=1}^{2} \log \ddet^{\ast}_{\Fr} \left( I + \frac{1}{2} \left( U_{\PP_{i}}^{-1} K_{i} + K_{i}^{-1} U_{\PP_{i}} \right) \right)
+ \log 2 \sum_{i=1}^{2} \zeta_{(I - \PP_{i}) (Q_{i} - {\mathcal A}_{i}) (I - \PP_{i})}(0) .
\end{eqnarray*}
\end{theorem}

\vspace{0.2 cm}

In the next two sections we are going to apply Theorem \ref{Theorem:2.5} to some boundary conditions satisfying the {\bf Condition A}.

\vspace{0.3 cm}

\section{Gluing formula of Dirac Laplacians with respect to $\PP_{-, {\mathcal L}_{0}}$ and $\PP_{+, {\mathcal L}_{1}}$}

\vspace{0.2 cm}

Let $(M, g)$ be an $m$-dimensional compact oriented Riemannian manifold with boundary $Y$ and $E \rightarrow M$ be a Hermitian flat vector bundle
with  a flat connection $\nabla$ which is compatible to the Hermitian structure on $E$.
We extend $\nabla$ to the de Rham operator acting on $E$-valued differential forms $\Omega^{\ast}(M, E)$, which we denote by $\nabla$ again.
We assume that near $Y$ $g$ is a product metric and $E$ has a product structure.
Using the Hodge star operator $\ast_{M}$, we define an involution $\Gamma : \Omega^{q}(M, E) \rightarrow \Omega^{m-q}(M, E)$ by

\begin{eqnarray}  \label{E:3.1}
\Gamma \omega & := &  i^{[\frac{m+1}{2}]} (-1)^{\frac{q(q+1)}{2}} \ast_{M} \omega \qquad \omega \in \Omega^{q}(M, E),
\end{eqnarray}

\vspace{0.2 cm}

\noindent
where $[\frac{m+1}{2}] = \frac{m}{2}$ for $m$ even and $\frac{m+1}{2}$ for $m$ odd.
Then $\Gamma^{2} = \Id$.
The odd signature operator $\B$ acting on $\Omega^{\bullet}(M, E)$ is defined by

\begin{eqnarray}  \label{E:3.2}
\B = \nabla \Gamma + \Gamma \nabla & : & \Omega^{\bullet}(M, E) \rightarrow \Omega^{\bullet}(M, E).
\end{eqnarray}

\noindent
Let $u$ be the normal coordinate to $Y$. A differential form $\omega$ is expressed near $Y$ by
$\omega = \omega_{\Tan} + du \wedge \omega_{\Nor}$,
where $\omega_{\Tan}$ and $\omega_{\Nor}$ are called the tangential and normal parts of $\omega$, respectively.
Using the product structure we can
induce a flat connection $\nabla^{Y} : \Omega^{\bullet}(Y, E|_{Y}) \rightarrow \Omega^{\bullet + 1}(Y, E|_{Y})$ from $\nabla$ and a Hodge
star operator $\ast_{Y} : \Omega^{\bullet}(Y, E|_{Y}) \rightarrow
\Omega^{m-1-\bullet}(Y, E|_{Y})$ from $\ast_{M}$.
We define two involutions $\beta$ and $\Gamma^{Y}$ by

\begin{eqnarray}   \label{E:3.3}
\beta & : & \Omega^{q}(Y, E|_{Y}) \rightarrow \Omega^{q}(Y, E|_{Y}), \qquad \beta(\omega) = (-1)^{q} \omega  \nonumber \\
\Gamma^{Y} & : & \Omega^{q}(Y, E|_{Y}) \rightarrow \Omega^{m-1-q}(Y, E|_{Y}),
\qquad \Gamma^{Y} (\omega) = i^{[\frac{m}{2}]} (-1)^{\frac{q(q+1)}{2}} \ast_{Y} \omega.
\end{eqnarray}

\noindent
Then $\beta^{2} = (\Gamma^{Y})^{2} = \Id$.
If we write
$\phi_{\Tan} + du \wedge \phi_{\Nor}$ by
$\left( \begin{array}{clcr} \phi_{\Tan} \\ \phi_{\Nor} \end{array} \right)$ near the boundary $Y$,
$\hspace{0.2 cm} \B$ is written by

\begin{eqnarray}  \label{E:3.4}
\B & = & \frac{1}{\sqrt{(-1)^{m}}} \hspace{0.1 cm} \beta \hspace{0.1 cm} \Gamma^{Y} \left( \begin{array}{clcr} 1 & 0 \\ 0 & 1 \end{array} \right) \left\{ \partial_{u}
\left( \begin{array}{clcr} 1 & 0 \\ 0 & 1 \end{array} \right) + (- \nabla^{Y} - \Gamma^{Y} \nabla^{Y} \Gamma^{Y})
\left( \begin{array}{clcr} 0 & 1 \\ 1 & 0 \end{array} \right) \right\}.
\end{eqnarray}

\noindent
Comparing (\ref{E:3.4}) with (\ref{E:2.1}), we have

\begin{eqnarray}  \label{E:3.5}
G = \frac{1}{\sqrt{(-1)^{m}}} \hspace{0.1 cm} \beta \hspace{0.1 cm} \Gamma^{Y} \left( \begin{array}{clcr} 1 & 0 \\ 0 & 1 \end{array} \right), \qquad
{\mathcal A} = - ( \nabla^{Y} + \Gamma^{Y} \nabla^{Y} \Gamma^{Y}) \left( \begin{array}{clcr} 0 & 1 \\ 1 & 0 \end{array} \right),
\end{eqnarray}

\noindent
which satisfy the relations (\ref{E:2.2}).

\vspace{0.2 cm}

We next describe the boundary conditions $\PP_{-, {\mathcal L}_{0}}$ and $\PP_{+, {\mathcal L}_{1}}$.
We put $\B_{Y} := \Gamma^{Y} \nabla^{Y} + \nabla^{Y} \Gamma^{Y}$.
Then ${\mathcal H}^{\bullet}(Y, E|_{Y}) := \Ker \B^{2}_{Y}$ is a finite dimensional vector space and we can decompose

\begin{eqnarray*}
\Omega^{\bullet}(Y, E|_{Y}) & = & \Imm \nabla^{Y} \oplus \Imm \Gamma^{Y} \nabla^{Y} \Gamma^{Y} \oplus {\mathcal H}^{\bullet}(Y, E|_{Y}).
\end{eqnarray*}

\noindent
If $\nabla \phi = \Gamma \nabla \Gamma \phi = 0$ for $\phi \in \Omega^{\bullet}(M, E)$, simple computation shows that $\phi$ is expressed on
$Y$ by

\begin{equation} \label{E:3.6}
\phi|_{Y} = \nabla^{Y} \varphi_{1} + \varphi_{2} + du \wedge ( \Gamma^{Y} \nabla^{Y} \Gamma^{Y} \psi_{1} + \psi_{2}),
\quad \varphi_{1}, \hspace{0.1 cm} \psi_{1} \in \Omega^{\bullet}(Y, E|_{Y}),
\quad \varphi_{2}, \hspace{0.1 cm} \psi_{2} \in {\mathcal H}^{\bullet}(Y, E|_{Y}).
\end{equation}

\noindent
Here $\varphi_{2}$ and $\psi_{2}$ are harmonic parts of $\hspace{0.1 cm} \iota^{\ast} \phi \hspace{0.1 cm}$ and
$\hspace{0.1 cm} \ast_{Y} \iota^{\ast} ( \ast_{M} \phi ) \hspace{0.1 cm}$ up to sign,
where $ \iota : Y \rightarrow M$ is the natural inclusion.
We define ${\mathcal K}$ by

\begin{equation} \label{E:3.7}
{\mathcal K} := \{ \varphi_{2} \in {\mathcal H}^{\bullet}(Y, E|_{Y}) \mid \nabla \phi = \Gamma \nabla \Gamma \phi = 0 \},
\end{equation}

\noindent
where $\phi$ has the form (\ref{E:3.6}).
If $\phi$ satisfies $\nabla \phi = \Gamma \nabla \Gamma \phi = 0$, so is $\Gamma \phi$ and hence

\begin{equation} \label{E:3.8}
\Gamma^{Y} {\mathcal K} = \{ \psi_{2} \in {\mathcal H}^{\bullet}(Y, E|_{Y}) \mid \nabla \phi = \Gamma \nabla \Gamma \phi = 0 \},
\end{equation}

\noindent
where $\phi$ has the form (\ref{E:3.6}).
Green formula (Corollary 2.3 in [9]) shows that
${\mathcal K}$ is perpendicular to $\Gamma^{Y} {\mathcal K}$.
We then have the following decomposition (cf. Corollary 8.4 in [13], Lemma 2.4 in [9]).

\begin{equation}  \label{E:3.9}
{\mathcal K} \oplus \Gamma^{Y} {\mathcal K} = {\mathcal H}^{\bullet}(Y, E|_{Y}),
\end{equation}

\noindent
which shows that
$( {\mathcal H}^{\bullet}(Y, E|_{Y}), \hspace{0.1 cm} \langle \hspace{0.1 cm}, \hspace{0.1 cm} \rangle_{Y}, \hspace{0.1 cm} \frac{1}{\sqrt{(-1)^{m}}} \beta \Gamma^{Y} )$
is a symplectic vector space with Lagrangian subspaces ${\mathcal K}$ and $\Gamma^{Y} {\mathcal K}$.
We denote by

\begin{equation}   \label{E:3.10}
{\mathcal L}_{0} = \left( \begin{array}{clcr} {\mathcal K} \\ {\mathcal K} \end{array} \right), \qquad
{\mathcal L}_{1} = \left( \begin{array}{clcr} \Gamma^{Y} {\mathcal K} \\ \Gamma^{Y} {\mathcal K} \end{array} \right).
\end{equation}

\vspace{0.2 cm}

We next define the orthogonal projections
${\mathcal P}_{-, {\mathcal L}_{0}}$, ${\mathcal P}_{+, {\mathcal L}_{1}} :
\left( \begin{array}{clcr} \Omega^{\bullet}(Y, E|_{Y})  \\ \oplus  \\ \Omega^{\bullet}(Y, E|_{Y}) \end{array} \right) \rightarrow
\left( \begin{array}{clcr} \Omega^{\bullet}(Y, E|_{Y}) \\  \oplus   \\ \Omega^{\bullet}(Y, E|_{Y}) \end{array} \right)$ by

\begin{eqnarray}    \label{E:3.11}
\Imm \PP_{-, {\mathcal L}_{0}}  =  \left( \begin{array}{clcr} \oplus_{q=0}^{m-1} \Omega^{q, -}(Y, E|_{Y}) \\
\oplus_{q=0}^{m-1} \Omega^{q, -}(Y, E|_{Y}) \end{array} \right) \oplus {\mathcal L}_{0},     \qquad
\Imm \PP_{+, {\mathcal L}_{1}} =  \left( \begin{array}{clcr} \oplus_{q=0}^{m-1} \Omega^{q, +}(Y, E|_{Y}) \\
\oplus_{q=0}^{m-1} \Omega^{q, +}(Y, E|_{Y}) \end{array} \right) \oplus {\mathcal L}_{1},
\end{eqnarray}

\vspace{0.2 cm}

\noindent
where $\hspace{0.1 cm} \Omega^{\bullet, -}(Y, E|_{Y}) := \Imm \nabla^{Y}$ and $\Omega^{\bullet, +}(Y, E|_{Y}) := \Imm \Gamma^{Y} \nabla^{Y} \Gamma^{Y}$.
Then ${\mathcal P}_{-, {\mathcal L}_{0}}$ and ${\mathcal P}_{+, {\mathcal L}_{1}}$ are pseudodifferential operators and give well-posed
boundary conditions for $\B$ and the refined analytic torsion (Lemma 2.15 in [9]).
The authors discussed the boundary problem of the refined analytic torsion on compact manifolds with boundary with these boundary conditions
in [9], [10], [11].
We denote by $\B_{{\mathcal P}_{-, {\mathcal L}_{0}}}$ and $\B^{2}_{q, {\mathcal P}_{-, {\mathcal L}_{0}}}$ the realizations of $\B$ and $\B^{2}_{q}$
with respect to ${\mathcal P}_{-, {\mathcal L}_{0}}$, {\it i.e.}

\begin{eqnarray}\label{E:3.12}
\Dom \left( \B_{{\mathcal P}_{-, {\mathcal L}_{0}}} \right) & = &
\left\{ \psi \in \Omega^{\bullet}(M, E) \mid {\mathcal P}_{-, {\mathcal L}_{0}} \left( \psi|_{Y} \right) = 0 \right\},  \nonumber\\
\Dom \left( \B^{2}_{q, {\mathcal P}_{-, {\mathcal L}_{0}}} \right) & = &
\left\{ \psi \in \Omega^{q}(M, E) \mid {\mathcal P}_{-, {\mathcal L}_{0}} \left( \psi|_{Y} \right) = 0, \hspace{0.1 cm}
{\mathcal P}_{-, {\mathcal L}_{0}} \left( (\B \psi)|_{Y} \right) = 0 \right\}.
\end{eqnarray}

\noindent
We define $\B_{{\mathcal P}_{+, {\mathcal L}_{1}}}$, $\B^{2}_{q, {\mathcal P}_{+, {\mathcal L}_{1}}}$ in the same way.
For $\psi = \psi_{\tan} + du \wedge \psi_{\Nor} \in \Omega^{q}(M, E)$, we define $\B^{2}_{q, \rel}$ and $\B^{2}_{q, \Abs}$ by

\begin{eqnarray}  \label{E:3.122}
\Dom \left( \B^{2}_{q, \rel} \right) & = &
\left\{ \psi \in \Omega^{q}(M, E) \mid \psi_{\tan}|_{Y} = 0, \hspace{0.1 cm} \left( \partial_{u} \psi_{\Nor} \right)|_{Y} = 0 \right\},  \nonumber\\
\Dom \left( \B^{2}_{q, \Abs} \right) & = &
\left\{ \psi \in \Omega^{q}(M, E) \mid \left( \partial_{u} \psi_{\tan} \right)|_{Y} = 0 , \hspace{0.1 cm} \psi_{\Nor}|_{Y} = 0 \right\}.
\end{eqnarray}

\noindent
The following result is straightforward (Lemma 2.11 in [9]).

\vspace{0.2 cm}

\begin{lemma}  \label{Lemma:3.1}
$$
\Ker \B^{2}_{q, {\mathcal P}_{-, {\mathcal L}_{0}}} = \ker \B^{2}_{q, \rel} = H^{q}(M, Y ; E), \qquad
\Ker \B^{2}_{q, {\mathcal P}_{+, {\mathcal L}_{1}}} = \ker \B^{2}_{q, \Abs} = H^{q}(M ; E).
$$
\end{lemma}

\vspace{0.2 cm}

We denote by $\left( \Omega^{\bullet}(M, E)|_{Y} \right)^{\ast}$ the orthogonal complement of
$\left( \begin{array}{clcr} {\mathcal H}^{\bullet}(Y, E|_{Y}) \\ {\mathcal H}^{\bullet}(Y, E|_{Y}) \end{array} \right)$ in
$\left( \Omega^{\bullet}(M, E)|_{Y} \right)$.
Then the action of the unitary operator $G$ splits according to the following decomposition.

\begin{eqnarray}  \label{E:3.13}
G : \left( \Omega^{\bullet}(M, E)|_{Y} \right)^{\ast} \oplus
\left( \begin{array}{clcr} {\mathcal H}^{\bullet}(Y, E|_{Y}) \\ {\mathcal H}^{\bullet}(Y, E|_{Y}) \end{array} \right)
\hspace{0.1 cm} \rightarrow \hspace{0.1 cm} \left( \Omega^{\bullet}(M, E)|_{Y} \right)^{\ast} \oplus
\left( \begin{array}{clcr} {\mathcal H}^{\bullet}(Y, E|_{Y}) \\ {\mathcal H}^{\bullet}(Y, E|_{Y}) \end{array} \right)
\end{eqnarray}

\vspace{0.2 cm}

\noindent
We define unitary maps $\hspace{0.1 cm} U_{\PP_{-}}$, $\hspace{0.1 cm} U_{\Pi_{>}} :
(\Omega^{\bullet}(M, E)|_{Y})^{\ast} \rightarrow (\Omega^{\bullet}(M, E)|_{Y})^{\ast} \hspace{0.1 cm}$ by

\begin{eqnarray}   \label{E:3.15}
U_{\PP_{-}} = (\B_{Y}^{2})^{- 1} \left( (\B_{Y}^{2})^{-} - ( \B_{Y}^{2})^{+} \right) \left( \begin{array}{clcr} 1 & 0 \\ 0 & 1 \end{array} \right), \quad
U_{\Pi_{>}} = (\B_{Y}^{2})^{- \frac{1}{2}} \left( \nabla^{Y} + \Gamma^{Y} \nabla^{Y} \Gamma^{Y} \right) \left( \begin{array}{clcr} 0 & - 1 \\ - 1 & 0 \end{array} \right),
\end{eqnarray}

\vspace{0.2 cm}

\noindent
where $(\B_{Y}^{2})^{-} := \nabla^{Y} \Gamma^{Y} \nabla^{Y} \Gamma^{Y}$, $(\B_{Y}^{2})^{+} := \Gamma^{Y} \nabla^{Y} \Gamma^{Y} \nabla^{Y}$
and $\B_{Y}^{2}$ is understood to be defined on $(\Omega^{\bullet}(M, E)|_{Y})^{\ast}$.
We denote the $\pm i$-eigenspace of $G$ in $(\Omega^{\bullet}(M, E)|_{Y} )^{\ast}$,
$\left( \begin{array}{clcr} {\mathcal H}^{\bullet}(Y, E|_{Y}) \\ {\mathcal H}^{\bullet}(Y, E|_{Y}) \end{array} \right)$ and
$\Omega^{\bullet}(M, E)|_{Y}$ by

\begin{eqnarray}    \label{E:3.14}
(\Omega^{\bullet}(M, E)|_{Y})^{\ast}_{\pm i} & :=  &
\frac{1}{2} (I \mp i G) \hspace{0.1 cm} (\Omega^{\bullet}(M, E)|_{Y})^{\ast}, \qquad
(\Ker {\mathcal A})_{\pm i} \hspace{0.1 cm} := \hspace{0.1 cm} \frac{1}{2} (I \mp i G) \hspace{0.1 cm}
\left( \begin{array}{clcr} {\mathcal H}^{\bullet}(Y, E|_{Y}) \\ {\mathcal H}^{\bullet}(Y, E|_{Y}) \end{array} \right),  \nonumber  \\
(\Omega^{\bullet}(M, E)|_{Y})_{\pm i} & :=  & (\Omega^{\bullet}(M, E))^{\ast}_{\pm i} \oplus (\Ker {\mathcal A})_{\pm i}.
\end{eqnarray}

\noindent
The following lemma is straightforward (cf. (3.2) - (3.5) and Lemma 3.1 in [10]).

\vspace{0.2 cm}

\begin{lemma} \label{Lemma:3.2}
(1) $\hspace{0.1 cm} U_{\PP_{-}}$ and $U_{\Pi_{>}}$ map $(\Omega^{\bullet}(M, E)|_{Y})^{\ast}_{\pm i}$ onto
$(\Omega^{\bullet}(M, E)|_{Y})^{\ast}_{\mp i}$. \newline
(2) $\hspace{0.1 cm} U_{\PP_{-}} U_{\PP_{-}} = U_{\Pi_{>}} U_{\Pi_{>}} = \Id$.
Hence, $U_{\PP_{-}}^{\ast} = U_{\PP_{-}}$ and $U_{\Pi_{>}}^{\ast} =  U_{\Pi_{>}}$.   \newline
(3) $\hspace{0.1 cm} U_{\Pi_{>}}^{\ast} U_{\PP_{-}} + U_{\PP_{-}}^{\ast} U_{\Pi_{>}} = 0$ . \newline
(4) $\hspace{0.1 cm} \Imm \PP_{-} = \{ \omega + U_{\PP_{-}} \omega \mid \omega \in (\Omega^{\bullet}(M, E)|_{Y})^{\ast}_{+i} \}$,
$\hspace{0.5 cm}\Imm \Pi_{>} = \{ \omega + U_{\Pi_{>}} \omega \mid \omega \in (\Omega^{\bullet}(M, E)|_{Y})^{\ast}_{+i} \}$.
\end{lemma}

\vspace{0.3 cm}

\noindent
We next choose a unitary map $\hspace{0.1 cm} U_{{\mathcal L}_{0}} : (\Ker {\mathcal A})_{+i} \rightarrow (\Ker {\mathcal A})_{-i} \hspace{0.1 cm}$
so that $\graph(U_{{\mathcal L}_{0}}) = \Imm {\mathcal L}_{0}$ and define
$U_{\PP_{-, {\mathcal L}_{0}}}$, $U_{\Pi_{>, {\mathcal L}_{0}}} : (\Omega^{\bullet}(M, E)|_{Y})_{+i}
\rightarrow (\Omega^{\bullet}(M, E)|_{Y})_{-i}$ by

\begin{eqnarray}  \label{E:3.16}
U_{\PP_{-, {\mathcal L}_{0}}} = \hspace{0.1 cm}  U_{\PP_{-}}|_{(\Omega^{\bullet}(M, E)|_{Y})^{\ast}_{+i}} + U_{{\mathcal L}_{0}}, \qquad
U_{\Pi_{>, {\mathcal L}_{0}}} = \hspace{0.1 cm}  U_{\Pi_{>}}|_{(\Omega^{\bullet}(M, E)|_{Y})^{\ast}_{+i}} + U_{{\mathcal L}_{0}}.
\end{eqnarray}

\vspace{0.2 cm}

\noindent
Then $\graph(U_{\PP_{-, {\mathcal L}_{0}}}) \hspace{0.1 cm} = \hspace{0.1 cm} \Imm \PP_{-, {\mathcal L}_{0}}$ and
$\graph(U_{\Pi_{>, {\mathcal L}_{0}}}) \hspace{0.1 cm} = \hspace{0.1 cm} \Imm \Pi_{>, {\mathcal L}_{0}}$.
By (\ref{E:3.5}) we have

\begin{eqnarray}    \label{E:3.19}
 \PP_{-, {\mathcal L}_{0}} \hspace{0.1 cm} {\mathcal A} \hspace{0.1 cm} \PP_{-, {\mathcal L}_{0}}  & = &
 \PP_{+, {\mathcal L}_{1}} \hspace{0.1 cm} {\mathcal A} \hspace{0.1 cm} \PP_{+, {\mathcal L}_{1}} = 0.
\end{eqnarray}

\noindent
Moreover, Theorem 2.1 in [14] shows that for $ t \geq 0$,

\begin{eqnarray}    \label{E:3.20}
Q(t) & = & \sqrt{{\mathcal A}^{2} + t} \hspace{0.2 cm}  + \hspace{0.2 cm} \text{a} \hspace{0.2 cm}  \text{smoothing} \hspace{0.2 cm} \text{operator},
\end{eqnarray}

\noindent
which together with (\ref{E:3.19}) shows that

\begin{eqnarray*}
\PP_{-, {\mathcal L}_{0}} \left( Q(t) - {\mathcal A} \right) \PP_{-, {\mathcal L}_{0}} & = & \PP_{-, {\mathcal L}_{0}} Q(t) \PP_{-, {\mathcal L}_{0}}
\hspace{0.1 cm} = \hspace{0.1 cm} \PP_{-, {\mathcal L}_{0}} \sqrt{{\mathcal A}^{2} + t} \hspace{0.1 cm} \PP_{-, {\mathcal L}_{0}}
\hspace{0.2 cm}  + \hspace{0.1 cm} \text{a} \hspace{0.1 cm}  \text{smoothing} \hspace{0.2 cm} \text{operator}.
\end{eqnarray*}

\vspace{0.2 cm}

\noindent
The same equality holds for $\PP_{+, {\mathcal L}_{1}}$. This shows that $\PP_{-, {\mathcal L}_{0}}$ and $\PP_{+, {\mathcal L}_{1}}$ satisfy the item (4)
in the {\bf Condition A}.
Since $\PP_{-, {\mathcal L}_{0}}$ and $\PP_{+, {\mathcal L}_{1}}$ are orthogonal pseudodifferential projections and $U_{\PP_{+, {\mathcal L}_{1}}} = - U_{\PP_{-, {\mathcal L}_{0}}}$,
the assertion (3) in Lemma \ref{Lemma:3.2} shows that
$\PP_{-, {\mathcal L}_{0}}$ and $\PP_{+, {\mathcal L}_{1}}$ satisfy the item (2), (3) in the {\bf Condition A} and hence satisfy the {\bf Condition A}.

Let $({\widehat M}, {\widehat g})$ be a closed Riemannian manifold and $Y$ be a hypersurface of ${\widehat M}$ such that
${\widehat M} - Y$ has two components, whose closures are denoted by $M_{1}$, $M_{2}$, {\it i.e.} ${\widehat M} = M_{1} \cup_{Y} M_{2}$.
We assume that ${\widehat g}$ is a product metric near $Y$.
We denote the odd signature operator on ${\widehat M}$ by $\B_{{\widehat M}}$ and its restriction to $M_{1}$ and $M_{2}$
by $\B_{M_{1}}$ and $\B_{M_{2}}$.
We now apply Theorem \ref{Theorem:2.5} with $\PP_{1} = \PP_{-, {\mathcal L}_{0}}$ and $\PP_{2} = I - \PP_{-, {\mathcal L}_{0}} = \PP_{+, {\mathcal L}_{1}}$.
Then we have the following equality.

\begin{eqnarray}   \label{E:3.17}
&  & \log \Det^{\ast} \B^{2}_{{\widehat M}} - \log \Det^{\ast} \B^{2}_{M_{1}, \PP_{-, {\mathcal L}_{0}}} -
\log \Det^{\ast} \B^{2}_{M_{2}, \PP_{+, {\mathcal L}_{1}}} \hspace{0.1 cm} = \hspace{0.1 cm}
- \log 2 \cdot ( \zeta_{{\mathcal A}^{2}}(0) + l ) + 2 \log \ddet A_{0}    \nonumber  \\
& - & 2 (\log \ddet V_{M_{1, \PP_{-, {\mathcal L}_{0}}}} + \log \ddet V_{M_{2, \PP_{+, {\mathcal L}_{1}}}}) +
2 \log | \ddet^{\ast}_{\Fr} \left( \frac{1}{2} \left( I - K_{1}^{-1} K_{2} \right) \right) |     \nonumber    \\
& - & \left\{ \log \ddet^{\ast}_{\Fr} \left( I + \frac{1}{2} \left( U_{\PP_{-, {\mathcal L}_{0}}}^{-1} K_{1} +
K_{1}^{-1} U_{\PP_{-, {\mathcal L}_{0}}} \right) \right) +
\log \ddet^{\ast}_{\Fr} \left( I - \frac{1}{2} \left( U_{\PP_{-, {\mathcal L}_{0}}}^{-1} K_{2} +
K_{2}^{-1} U_{\PP_{-, {\mathcal L}_{0}}} \right) \right)  \right\}
\nonumber\\
& + &  \log 2 \left( \zeta_{\left( \PP_{-, {\mathcal L}_{0}} (Q_{1} - {\mathcal A}_{1}) \PP_{-, {\mathcal L}_{0}}) \right)}(0) +
\zeta_{\left( \PP_{+, {\mathcal L}_{1}} (Q_{2} - {\mathcal A}_{2}) \PP_{+, {\mathcal L}_{1}} \right)}(0) \right) ,
\end{eqnarray}

\vspace{0.2 cm}

\noindent
where $K_{i} : \left( \Omega^{\bullet}(M_{i}, E)|_{Y} \right)_{+i} \rightarrow \left( \Omega^{\bullet}(M_{i}, E)|_{Y} \right)_{-i}$
is a unitary operator such that
$\graph (K_{i}) = \Imm {\mathcal C}_{i}$, the Cauchy data space with respect to $\B_{M_{i}}$.
By Lemma \ref{Lemma:3.1} we have

\begin{eqnarray}  \label{E:3.25}
\log \ddet A_{0} & = &  \sum_{q=0}^{m} \log \ddet A_{0, q},    \\
\log \ddet V_{M_{1}, \PP_{-, {\mathcal L}_{0}}} & = & \sum_{q=0}^{m} \log \ddet V_{M_{1}, q, \rel} , \qquad
\log \ddet V_{M_{2}, \PP_{+, {\mathcal L}_{1}}} \hspace{0.1 cm} =  \hspace{0.1 cm}  \sum_{q=0}^{m} \log \ddet V_{M_{2}, q, \Abs}, \nonumber
\end{eqnarray}

\noindent
where $A_{0, q}$ is the Hermitian matrix obtained by simply replacing ${\mathcal D}^{2}_{{\widehat M}}$ in (\ref{E:2.15})
with $\B^{2}_{{\widehat M}, q}$ acting on $\Omega^{q}({\widehat M}, {\widehat E})$.
Similarly, $V_{M_{i}, q, \rel/\Abs}$ is the Hermitian matrix obtained by replacing ${\mathcal D}^{2}_{M, P}$ in (\ref{E:2.12}) with
$\B^{2}_{M_{i}, q, \rel/\Abs}$ acting on $\Omega^{q}(M_{i}, E)$ satisfying the relative/absolute boundary conditions.
Lemma 2.5 in [15] and Lemma \ref{Lemma:3.1} show that

\begin{eqnarray}   \label{E:3.18}
\Dim \Ker \left( \PP_{-, {\mathcal L}_{0}} (Q_{1} - {\mathcal A}_{1}) \PP_{-, {\mathcal L}_{0}}) \right) & = &
\sum_{q=0}^{m} \beta_{q}(M_{1}) = \sum_{q=0}^{m} \beta_{q}(M_{1}, Y),  \nonumber   \\
\Dim \Ker \left( \PP_{+, {\mathcal L}_{1}} (Q_{2} - {\mathcal A}_{2}) \PP_{+, {\mathcal L}_{1}} \right) & = &
\sum_{q=0}^{m} \beta_{q}(M_{2}, Y) = \sum_{q=0}^{m} \beta_{q}(M_{2}) ,
\end{eqnarray}

\noindent
where $\beta_{q}(M_{i}, Y) := \Dim H^{q}(M_{i}, Y ; E)$ and $\beta_{q}(M_{i}) := \Dim H^{q}(M_{i} ; E)$.
By (\ref{E:3.19}) and (\ref{E:3.18}), we have

\begin{eqnarray}    \label{E:3.21}
\zeta_{\left( \PP_{-, {\mathcal L}_{0}} Q_{1} \PP_{-, {\mathcal L}_{0}} \right)}(0) +
\Dim \Ker \left( \PP_{-, {\mathcal L}_{0}} Q_{1} \PP_{-, {\mathcal L}_{0}} \right)  & = &
\zeta_{\PP_{-, {\mathcal L}_{0}} \sqrt{{\mathcal A}_{1}^{2}} \PP_{-, {\mathcal L}_{0}}}(0)
+ \Dim \Ker \PP_{-, {\mathcal L}_{0}} \sqrt{{\mathcal A}_{1}^{2}} \PP_{-, {\mathcal L}_{0}} \nonumber \\
& = & \sum_{q=0}^{m-1} \left( \zeta_{\B_{Y, q}^{2}}(0) + \beta_{q}(Y) \right),
\end{eqnarray}

\noindent
where $\beta_{q}(Y) := \Dim \Ker H^{q}(Y ; E|_{Y})$.
Similarly, we have

\begin{eqnarray}   \label{E:3.22}
 \zeta_{\left( \PP_{+, {\mathcal L}_{1}} Q_{2} \PP_{+, {\mathcal L}_{1}} \right)}(0) +
\Dim \Ker \left( \PP_{+, {\mathcal L}_{1}} Q_{2} \PP_{+, {\mathcal L}_{1}} \right)   & = &
\sum_{q=0}^{m-1} \left( \zeta_{\B_{Y, q}^{2}}(0) + \beta_{q}(Y) \right).
\end{eqnarray}

\noindent
On the other hand,

\begin{eqnarray}    \label{E:3.23}
\zeta_{{\mathcal A}^{2}}(0) + l & = & 2 \sum_{q=0}^{m-1} \left( \zeta_{\B_{Y, q}^{2}}(0) + \beta_{q}(Y) \right).
\end{eqnarray}

\noindent
Summarizing the above argument, we have the following result, which is the main result of this section.

\vspace{0.2 cm}

\begin{theorem} \label{Theorem:1.1}
Let $({\widehat M}, {\widehat g})$ be a closed Riemannian manifold and $Y$ be a hypersurface of ${\widehat M}$ with ${\widehat M} = M_{1} \cup_{Y} M_{2}$.
We assume that ${\widehat g}$ is a product metric near $Y$. Then :

\begin{eqnarray*}
&  & \sum_{q=0}^{m} \left( \log \Det^{\ast} \B_{{\widehat M}, q}^{2} - \log \Det^{\ast} \B^{2}_{M_{1}, q, \PP_{-, {\mathcal L}_{0}}} -
\log \Det^{\ast} \B^{2}_{M_{2}, q, \PP_{+, {\mathcal L}_{1}}} \right)   \\
& = &  - \log 2 \cdot \sum_{q=0}^{m} \left( \beta_{q}(M_{1}) + \beta_{q}(M_{2}) \right) + 2 \sum_{q=0}^{m} \log \ddet A_{0, q}
 -    2 \sum_{q=0}^{m} \left( \log \ddet V_{M_{1, q, \rel}} + \log \ddet V_{M_{2, q, \Abs}} \right)    \\
& &  + \hspace{0.1 cm}  2 \log | \ddet^{\ast}_{\Fr} \left( \frac{1}{2} \left( I - K_{1}^{-1} K_{2} \right) \right) |
 -    \log \ddet^{\ast}_{\Fr} \left( I + \frac{1}{2} \left( U_{\PP_{-, {\mathcal L}_{0}}}^{-1} K_{1} \hspace{0.1 cm} + \hspace{0.1 cm}
K_{1}^{-1} U_{\PP_{-, {\mathcal L}_{0}}} \right) \right) \\
& &  - \hspace{0.1 cm}
\log \ddet^{\ast}_{\Fr} \left( I - \frac{1}{2} \left( U_{\PP_{-, {\mathcal L}_{0}}}^{-1} K_{2} +
K_{2}^{-1} U_{\PP_{-, {\mathcal L}_{0}}} \right) \right)  .
\end{eqnarray*}
\end{theorem}

\noindent
{\it Remark} : The kernel of $ \left( I + \frac{1}{2} \left( U_{\PP_{-, {\mathcal L}_{0}}}^{-1} K_{1} \hspace{0.1 cm} + \hspace{0.1 cm}
K_{1}^{-1} U_{\PP_{-, {\mathcal L}_{0}}} \right) \right)$ consists of $\omega \in \left( \Omega^{\bullet}(M_{1}, E)|_{Y} \right)_{+i}$
such that $\omega -  U_{\PP_{-, {\mathcal L}_{0}}} \omega \hspace{0.1 cm} ( \hspace{0.1 cm} = \omega +  K_{1} \omega )$ can be extended to a solution of
$\B_{M_{1}, \PP_{-, {\mathcal L}_{0}}}$. The same result holds for $ \left( I - \frac{1}{2} \left( U_{\PP_{-, {\mathcal L}_{0}}}^{-1} K_{2} \hspace{0.1 cm} + \hspace{0.1 cm} K_{2}^{-1} U_{\PP_{-, {\mathcal L}_{0}}} \right) \right)$.

\vspace{0.3 cm}

\section{Gluing formula of Dirac Laplacians with respect to the absolute and relative boundary conditions}

\vspace{0.2 cm}

We continue to use the same notations as in the previous section. In this section we consider a double of de Rham complexes
$\hspace{0.1 cm} \Omega^{\bullet}(M, E \oplus E) := \Omega^{\bullet}(M, E) \oplus \Omega^{\bullet}(M, E)$, which was used in [24].
We define the odd signature operator ${\widetilde \B}$ and a boundary condition ${\widetilde \PP}$ in this context  as follows.

\begin{eqnarray}    \label{E:4.1}
{\widetilde \B} & = & \left( \begin{array}{clcr} 0 & \B \\ \B & 0 \end{array} \right) =
\left( \begin{array}{clcr} 0 & \Gamma \nabla + \nabla \Gamma \\ \Gamma \nabla + \nabla \Gamma & 0 \end{array} \right)  :
\Omega^{\bullet}(M, E \oplus E) \rightarrow \Omega^{\bullet}(M, E \oplus E)    \nonumber \\
{\widetilde \PP} & = &  \left( \begin{array}{clcr} \PP_{\rel} & 0 \\ 0 & \PP_{\Abs} \end{array} \right)  :
\Omega^{\bullet}(M, E \oplus E)|_{Y} \rightarrow \Omega^{\bullet}(M, E \oplus E)|_{Y},
\end{eqnarray}

\noindent
where $\PP_{\rel}$ and $\PP_{\Abs}$ are orthogonal projections defined by

\begin{eqnarray}    \label{E:4.2}
\PP_{\rel} (\omega_{\Tan}|_{Y} + du \wedge \omega_{\Nor}|_{Y}) = \omega_{\Tan}|_{Y}, \qquad
\PP_{\Abs} (\omega_{\Tan}|_{Y} + du \wedge \omega_{\Nor}|_{Y}) = \omega_{\Nor}|_{Y}.
\end{eqnarray}

\noindent
Then the realization ${\widetilde \B}^{2}_{{\widetilde \PP}}$ with respect to the boundary condition
${\widetilde \PP}$ is given as follows.

\begin{eqnarray}    \label{E:4.3}
\Dom \left( {\widetilde \B}^{2}_{{\widetilde \PP}} \right) & = & \left\{ \left( \begin{array}{clcr} \phi \\ \psi \end{array} \right)
\in \Omega^{\bullet}(M, E \oplus E) \mid  {\widetilde \PP} \left( \begin{array}{clcr} \phi|_{Y} \\ \psi|_{Y} \end{array} \right) = 0, \quad
{\widetilde \PP} \left( {\widetilde \B} \left( \begin{array}{clcr} \phi \\ \psi \end{array} \right)|_{Y} \right)  = 0 \right\}  \nonumber \\
& = & \left\{ \left( \begin{array}{clcr} \phi_{1} + du \wedge \phi_{2} \\ \psi_{1} + du \wedge \psi_{2} \end{array} \right)
 \mid \hspace{0.1 cm}  \phi_{1}|_{Y} = 0, \hspace{0.1 cm} (\partial_{u} \phi_{2})|_{Y} = 0, \hspace{0.1 cm} (\partial_{u} \psi_{1})|_{Y} = 0,
 \hspace{0.1 cm} \psi_{2}|_{Y} = 0,  \hspace{0.1 cm} \right\}.
\end{eqnarray}

\noindent
By (\ref{E:3.122}) and the Poincar\'e duality we have

\begin{eqnarray}    \label{E:4.4}
{\widetilde \B}^{2}_{{\widetilde \PP}} & = &
\left( \begin{array}{clcr} \B^{2}_{M, \rel} & 0 \\ 0 & \B^{2}_{M, \Abs} \end{array} \right),   \\
\log \Det^{\ast} {\widetilde \B}^{2}_{{\widetilde \PP}} & = & \sum_{q=0}^{m} \left( \log \Det^{\ast} \B^{2}_{M, q, \rel} + \log \Det^{\ast} \B^{2}_{M, q, \Abs} \right)
\hspace{0.1 cm} = \hspace{0.1 cm}  2 \sum_{q=0}^{m} \log \Det^{\ast} \B^{2}_{M, q, \rel}.    \nonumber
\end{eqnarray}

\vspace{0.2 cm}

We put

\begin{eqnarray}    \label{E:4.5}
I = \left( \begin{array}{clcr} 1 & 0 \\ 0 & 1 \end{array} \right), \qquad L = \left( \begin{array}{clcr} 0 & 1 \\ 1 & 0 \end{array} \right), \qquad
S = \left( \begin{array}{clcr} 1 & 0 \\ 0 & -1 \end{array} \right).
\end{eqnarray}

\noindent
If we write
$\left( \begin{array}{clcr} \phi_{1} + du \wedge \phi_{2} \\ \psi_{1} + du \wedge \psi_{2} \end{array} \right)$ by
$\left( \begin{array}{clcr} \phi_{1} \\ \phi_{2} \\ \psi_{1} \\ \psi_{2} \end{array} \right)$,
$\hspace{0.2 cm} {\widetilde \B}$ is written, near the boundary $Y$, by

\begin{eqnarray}  \label{E:4.6}
{\widetilde \B} & = & \frac{1}{\sqrt{(-1)^{m}}} \hspace{0.1 cm} \beta \hspace{0.1 cm} \Gamma^{Y} \left( \begin{array}{clcr} 0 & I \\ I & 0 \end{array} \right)
\left\{ \partial_{u} - \left( \nabla^{Y} + \Gamma^{Y} \nabla^{Y} \Gamma^{Y} \right)
\left( \begin{array}{clcr} L & 0 \\ 0 & L \end{array} \right) \right\}
\hspace{0.1 cm} = \hspace{0.1 cm}  {\widetilde G} \left( \partial_{u} + {\widetilde {\mathcal A}} \right).
\end{eqnarray}

\noindent
Comparing (\ref{E:4.6}) with (\ref{E:2.1}), we have

\begin{eqnarray}  \label{E:4.7}
{\widetilde G} = \frac{1}{\sqrt{(-1)^{m}}} \hspace{0.1 cm} \beta \hspace{0.1 cm} \Gamma^{Y}
\left( \begin{array}{clcr} 0 & I \\ I & 0 \end{array} \right), \qquad
{\widetilde {\mathcal A}} = - \left( \nabla^{Y} + \Gamma^{Y} \nabla^{Y} \Gamma^{Y} \right)
\left( \begin{array}{clcr} L & 0 \\ 0 & L \end{array} \right),
\end{eqnarray}

\noindent
which satisfy the relations in (\ref{E:2.2}).
We denote by ${\widetilde \Pi}_{>} := \Pi_{>} \oplus \Pi_{>}$ the orthogonal projection onto the space spanned by positive eigenforms of
${\widetilde {\mathcal A}}$.
We denote the $\pm i$-eigenspace of ${\widetilde G}$ by

\begin{eqnarray}    \label{E:4.8}
(\Omega^{\bullet}(M, E \oplus E)|_{Y})_{\pm i} & := & \frac{1}{2} (I \mp i {\widetilde G})
\left( \Omega^{\bullet}(M, E)|_{Y} \oplus \Omega^{\bullet}(M, E)|_{Y} \right).
\end{eqnarray}

\noindent
For instance, if $m$ is odd, simple computation shows that

\begin{eqnarray}   \label{E:4.9}
(\Omega^{\bullet}(M, E \oplus E)|_{Y})_{+i}  =
\Span \left( \begin{array}{clcr} \omega_{1} \\ \omega_{2} \\  - \beta \Gamma^{Y} \omega_{1} \\  - \beta \Gamma^{Y} \omega_{2} \end{array} \right), \hspace{0.1 cm}
(\Omega^{\bullet}(M, E \oplus E)|_{Y})_{-i}  \hspace{0.1 cm} = \hspace{0.1 cm}
\Span \left( \begin{array}{clcr} \omega_{1} \\ \omega_{2} \\  \beta \Gamma^{Y} \omega_{1} \\  \beta \Gamma^{Y} \omega_{2} \end{array} \right),
\end{eqnarray}

\vspace{0.2 cm}

\noindent
where $\omega_{1}$, $\omega_{2} \in \Omega^{\bullet}(Y, E|_{Y})$. This fact will be used in (\ref{E:4.24}) below.
Like (\ref{E:3.13}), we write

\begin{eqnarray}  \label{E:4.10}
\Omega^{\bullet}(M, E \oplus E)|_{Y}  & = &
\left( \Omega^{\bullet}(M, E \oplus E)|_{Y} \right)^{\ast} \oplus \Ker {\widetilde {\mathcal A}}, \quad
\left( \Omega^{\bullet}(M, E \oplus E)|_{Y} \right)^{\ast} =: (\Ker {\widetilde {\mathcal A}})^{\perp}.
\end{eqnarray}

\noindent
We define unitary maps $\hspace{0.1 cm} U_{{\widetilde \PP}}$,
$\hspace{0.1 cm} U_{{\widetilde \Pi}_{>}} :
\left( \Omega^{\bullet}(M, E \oplus E)|_{Y} \right)^{\ast} \rightarrow
\left( \Omega^{\bullet}(M, E \oplus E)|_{Y} \right)^{\ast} \hspace{0.1 cm}$ by  ((\ref{E:3.15}))

\begin{eqnarray}    \label{E:4.11}
U_{{\widetilde \PP}} =  \sqrt{(-1)^{m+1}} \hspace{0.1 cm} \beta \hspace{0.1 cm} \Gamma^{Y} \left( \begin{array}{clcr} 0 & - S \\ S & 0 \end{array} \right), \qquad
U_{{\widetilde \Pi}_{>}} = (\B_{Y}^{2})^{-\frac{1}{2}} \left( \nabla^{Y} + \Gamma^{Y} \nabla^{Y} \Gamma^{Y} \right)
\left( \begin{array}{clcr} - L & 0 \\ 0 & - L \end{array} \right).
\end{eqnarray}

\noindent
Here the domain of $U_{{\widetilde \PP}}$ can be naturally extended to $\Omega^{\bullet}(M, E \oplus E)|_{Y} $.
The following lemma is an analogue of Lemma \ref{Lemma:3.2}, whose proof is straightforward.

\begin{lemma} \label{Lemma:4.1}
(1) $\hspace{0.1 cm} U_{{\widetilde \PP}}$ and $U_{{\widetilde \Pi}_{>}}$ map $\left( \Omega^{\bullet}(M, E \oplus E)|_{Y} \right)^{\ast}_{\pm i}$ onto
$\left( \Omega^{\bullet}(M, E \oplus E)|_{Y} \right)^{\ast}_{\mp i}$. \newline
(2) $\hspace{0.1 cm} U_{{\widetilde \PP}} U_{{\widetilde \PP}} = - \Id$, $\hspace{0.1 cm} U_{{\widetilde \Pi}_{>}} U_{{\widetilde \Pi}_{>}} = \Id$ and
$\hspace{0.1 cm} U_{{\widetilde \Pi}_{>}} U_{{\widetilde \PP}} \hspace{0.1 cm} = \hspace{0.1 cm} U_{{\widetilde \PP}} U_{{\widetilde \Pi}_{>}}$.
Hence, $U_ {{\widetilde \PP}}^{\ast} = - U_{{\widetilde \PP}}$ and $U_{{\widetilde \Pi}_{>}}^{\ast} =  U_{{\widetilde \Pi}_{>}}$.  \newline
(3) $\hspace{0.1 cm} U_{{\widetilde \Pi}_{>}}^{\ast} U_{{\widetilde \PP}} + U_{{\widetilde \PP}}^{\ast} U_{{\widetilde \Pi}_{>}} = 0$. \newline
(4) $\hspace{0.1 cm} \Imm {\widetilde \PP} = \{ \omega + U_{{\widetilde \PP}} \omega \mid \omega \in
\left( \Omega^{\bullet}(M, E \oplus E)|_{Y} \right)_{+i} \} \hspace{0.1 cm}$,
$\hspace{0.2 cm} \Imm {\widetilde \Pi}_{>} = \{ \omega + U_{{\widetilde \Pi}_{>}} \omega \mid \omega \in
\left( \Omega^{\bullet}(M, E \oplus E)|_{Y} \right)^{\ast}_{+i} \}$.
\end{lemma}

\vspace{0.2 cm}

\noindent
{\it Remark} : It is not difficult to see that there is no unitary map from $(\Omega^{\bullet}(M, E)|_{Y})_{+i}$ to
$(\Omega^{\bullet}(M, E)|_{Y})_{-i}$ whose graph is $\Imm \PP_{\rel}$ or $\Imm \PP_{\Abs}$.
Hence, we cannot apply Theorem \ref{Theorem:2.5} to this case. This is the reason why we consider the double of
de Rham complexes as above.

\vspace{0.2 cm}

\noindent
It is straightforward that

\begin{eqnarray}  \label{E:4.12}
{\widetilde \PP} \hspace{0.1 cm} {\widetilde {\mathcal A}} \hspace{0.1 cm} {\widetilde \PP}
 \hspace{0.1 cm} = \hspace{0.1 cm} (I - {\widetilde \PP}) \hspace{0.1 cm} {\widetilde {\mathcal A}} \hspace{0.1 cm} (I - {\widetilde \PP})
  \hspace{0.1 cm} = \hspace{0.1 cm} 0,
\end{eqnarray}

\noindent
and by Theorem 2.1 in [14] (cf. (\ref{E:3.20})) we have

\begin{eqnarray}    \label{E:4.13}
\hspace{0.2 cm} {\widetilde Q}(t)  & = &  \left( \begin{array}{clcr} Q(t) & 0 \\ 0 & Q(t) \end{array} \right)
\hspace{0.2 cm}  = \hspace{0.2 cm}
\left( \begin{array}{clcr} \sqrt{{\mathcal A}^{2} + t} & 0 \\ 0 & \sqrt{{\mathcal A}^{2} + t} \end{array} \right)
\hspace{0.2 cm}  + \hspace{0.2 cm} \text{a} \hspace{0.2 cm}  \text{smoothing} \hspace{0.2 cm} \text{operator},
\end{eqnarray}

\noindent
which shows that ${\widetilde \PP}$ and $I - {\widetilde \PP}$ satisfy the item (4) in the {\bf Condition A}.
This fact and the assertion (3) in Lemma \ref{Lemma:4.1} show that ${\widetilde \PP}$ and $I - {\widetilde \PP}$ satisfy the {\bf Condition A},
as in the previous section,

We next consider a partitioned manifold ${\widehat M} = M_{1} \cup_{Y} M_{2}$ as before.
We assume the same assumptions as in Theorem \ref{Theorem:2.5}.
Let ${\widetilde K}_{i} : \left( \Omega^{\bullet}(M_{i}, E \oplus E)|_{Y} \right)_{+i} \rightarrow
\left( \Omega^{\bullet}(M_{i}, E \oplus E)|_{Y} \right)_{-i}$ be a unitary operator such that
$\graph ({\widetilde K}_{i}) = \Imm {\mathcal {\widetilde C}}_{i}$, the Cauchy data space with respect to ${\widetilde \B}_{M_{i}}$.
We denote by ${\widetilde Q}_{i}$ the Neumann jump operator for ${\widetilde \B}^{2}_{M_{i}}$ on $M_{i}$ and
by ${\widetilde {\mathcal A}}_{i}$ the tangential Dirac operator of ${\widetilde \B}_{M_{i}}$.
We now apply Theorem \ref{Theorem:2.5} with $\PP_{1} = {\widetilde \PP}$ and $\PP_{2} = I - {\widetilde \PP}$.
Since $U_{I - {\widetilde \PP}} = - U_{{\widetilde \PP}}$, we have the following equality.

\begin{eqnarray}   \label{E:4.14}
&  & \log \Det^{\ast} {\widetilde \B}^{2}_{{\widehat M}} - \log \Det^{\ast} {\widetilde \B}^{2}_{M_{1}, {\widetilde \PP}} -
\log \Det^{\ast} {\widetilde \B}_{M_{2}, I - {\widetilde \PP}} \hspace{0.1 cm} = \hspace{0.1 cm} -
\log 2 \cdot ( \zeta_{{\widetilde {\mathcal A}}^{2}}(0) + {\widetilde l} )  \nonumber \\
& + &   2 \log \ddet {\widetilde A}_{0}
 - 2 \left( \log \ddet {\widetilde V}_{M_{1}, {\widetilde \PP}} + \log \ddet {\widetilde V}_{M_{2}, I - {\widetilde \PP}} \right) +
 2 \log | \ddet^{\ast}_{\Fr} \left( \frac{1}{2} \left( I - {\widetilde K}_{1}^{-1} {\widetilde K}_{2} \right) \right) |   \nonumber\\
& - & \left\{ \log \ddet^{\ast}_{\Fr} \left( I + \frac{1}{2} \left( U_{{\widetilde \PP}}^{-1} {\widetilde K}_{1} + {\widetilde K}_{1}^{-1} U_{{\widetilde \PP}} \right) \right)  +
\log \ddet^{\ast}_{\Fr} \left( I - \frac{1}{2} \left( U_{{\widetilde \PP}}^{-1} {\widetilde K}_{2} + {\widetilde K}_{2}^{-1} U_{{\widetilde \PP}} \right) \right)
\right\}   \nonumber \\
& + & \log 2 \cdot \left( \zeta_{\left( ( I - {\widetilde \PP} ) ({\widetilde Q}_{1} - {\widetilde {\mathcal A}}_{1}) ( I - {\widetilde \PP}) \right)}(0) +
\zeta_{\left( {\widetilde \PP} ({\widetilde Q}_{2} - {\widetilde {\mathcal A}}_{2}) {\widetilde \PP} \right)}(0) \right) .
\end{eqnarray}

\vspace{0.2 cm}

\noindent
With the same notations in (\ref{E:3.25}), we note that

\begin{eqnarray}   \label{E:4.15}
\log \ddet {\widetilde A}_{0} & = & 2 \sum_{q=0}^{m} \log \ddet A_{0, q},    \\
\log \ddet {\widetilde V}_{M_{i}, {\widetilde \PP}} & = &
\log \ddet {\widetilde V}_{M_{i}, I - {\widetilde \PP}} \hspace{0.1 cm} = \hspace{0.1 cm}
\sum_{q=0}^{m} \left( \log \ddet V_{M_{i}, q, \rel} + \log \ddet V_{M_{i}, q, \Abs} \right)  \nonumber  \\
& = & 2 \sum_{q=0}^{m} \log \ddet V_{M_{i}, q, \rel} \hspace{0.1 cm} = \hspace{0.1 cm}
2 \sum_{q=0}^{m} \log \ddet V_{M_{i}, q, \Abs}. \nonumber
\end{eqnarray}

\noindent
Lemma 2.5 in [15] shows that

\begin{eqnarray}    \label{E:4.16}
\Dim \Ker \left( {\widetilde \PP} ({\widetilde Q}_{i} - {\widetilde {\mathcal A}}_{i}) {\widetilde \PP}) \right)  \hspace{0.1 cm} = \hspace{0.1 cm}
\Dim \Ker \left( (I - {\widetilde \PP}) ({\widetilde Q}_{i} - {\widetilde {\mathcal A}}_{i}) (I - {\widetilde \PP}) \right)   \hspace{0.1 cm} = \hspace{0.1 cm}
 \sum_{q=0}^{m} \left( \beta_{q}(M_{i}) + \beta_{q}(M_{i}, Y) \right) 
\end{eqnarray}

\noindent
Since
$\hspace{0.1 cm} \Imm {\widetilde \PP} = \Imm ( I - {\widetilde \PP}) = \oplus_{q=0}^{m-1} \left( \Omega^{q}(Y, E|_{Y}) \oplus \Omega^{q}(Y, E|_{Y}) \right)
\hspace{0.1 cm}$, the equalities (\ref{E:4.12}) and (\ref{E:4.13}) lead to

\begin{eqnarray}    \label{E:4.18}
& & \zeta_{\left( {\widetilde \PP} {\widetilde Q}_{i} {\widetilde \PP} \right)}(0) +
\Dim \Ker \left( {\widetilde \PP} {\widetilde Q}_{i} {\widetilde \PP} \right) =
\zeta_{\left( (I - {\widetilde \PP}) {\widetilde Q}_{i} (I - {\widetilde \PP}) \right)}(0) +
\Dim \Ker \left( (I - {\widetilde \PP}) {\widetilde Q}_{i} (I - {\widetilde \PP}) \right)   \nonumber  \\
& = & 2 \hspace{0.1 cm} \left( \zeta_{\sqrt{\B^{2}_{Y}}}(0) + \Dim \Ker \sqrt{\B^{2}_{Y}} \right)
\hspace{0.1 cm} = \hspace{0.1 cm} 2 \sum_{q=0}^{m-1} \left( \zeta_{\B^{2}_{Y, q}}(0) + \Dim \Ker \B^{2}_{Y, q} \right).
\end{eqnarray}

\noindent
Hence, by (\ref{E:4.16}) and (\ref{E:4.18}) we have

\begin{eqnarray}    \label{E:4.19}
 \zeta_{\left( {\widetilde \PP} ({\widetilde Q}_{1} - {\widetilde {\mathcal A}}_{1}) {\widetilde \PP}) \right)}(0) +
\zeta_{\left( (I - {\widetilde \PP}) ({\widetilde Q}_{2} - {\widetilde {\mathcal A}}_{2}) (I - {\widetilde \PP}) \right)}(0)
& = & \hspace{0.1 cm} 4 \sum_{q=0}^{m-1} \left( \zeta_{\B^{2}_{Y, q}}(0) +  \Dim \Ker \B^{2}_{Y, q} \right)   \nonumber \\
& & - \hspace{0.1 cm} 2 \sum_{q=0}^{m} \left( \beta_{q}(M_{1}) + \beta_{q}(M_{2}) \right).
\end{eqnarray}

\noindent
Similarly, by (\ref{E:4.7}) we have

\begin{eqnarray}  \label{E:4.20}
{\widetilde {\mathcal A}}^{2} \hspace{0.1 cm} = \hspace{0.1 cm}  \B^{2}_{Y} \left( \begin{array}{clcr}  I & 0 \\ 0 &  I \end{array} \right), \qquad \text{and} \qquad
\zeta_{{\widetilde {\mathcal A}}^{2}}(0) + {\widetilde l} \hspace{0.1 cm} = \hspace{0.1 cm} 4 \sum_{q=0}^{m-1} \left( \zeta_{\B^{2}_{Y, q}}(0) +  \Dim \Ker \B^{2}_{Y, q} \right).
\end{eqnarray}

\noindent
Hence (\ref{E:4.14}) can be rewritten as follows.

\begin{eqnarray}   \label{E:4.21}
&  & \sum_{q=0}^{m} \left( \log \Det^{\ast} \B^{2}_{{\widehat M}, q} - \log \Det^{\ast} \B^{2}_{M_{1}, q, \rel} - \log \Det^{\ast} \B^{2}_{M_{2}, q, \Abs} \right)  \hspace{0.1 cm}  =  \hspace{0.1 cm}  - \log 2 \sum_{q=0}^{m} \left( \beta_{q}(M_{1}) + \beta_{q}(M_{2}) \right)   \nonumber   \nonumber \\
& + &    2 \sum_{q=0}^{m} \log \ddet A_{0, q} \hspace{0.1 cm} - \hspace{0.1 cm} 2
\sum_{q=0}^{m} \left( \log \ddet V_{M_{1, q, \rel}} + \log \ddet V_{M_{2, q, \Abs}} \right)
\hspace{0.1 cm}  + \hspace{0.1 cm} \log | \ddet^{\ast}_{\Fr} \left( \frac{1}{2} \left( I - {\widetilde K}_{1}^{-1} {\widetilde K}_{2} \right) \right) |   \nonumber \\
& - & \frac{1}{2} \left\{ \log \ddet^{\ast}_{\Fr} \left( I + \frac{1}{2}
\left( U_{{\widetilde \PP}}^{-1} {\widetilde K}_{1} + {\widetilde K}_{1}^{-1} U_{{\widetilde \PP}} \right) \right) +
\log \ddet^{\ast}_{\Fr} \left( I - \frac{1}{2} \left( U_{{\widetilde \PP}}^{-1} {\widetilde K}_{2} + {\widetilde K}_{2}^{-1} U_{{\widetilde \PP}} \right) \right) \right\}.
\end{eqnarray}

\vspace{0.2 cm}

Finally, we analyze the last three terms in (\ref{E:4.21}).
We discuss only the case when the dimension of ${\widehat M}$ is odd.
The same method can be used for an even dimensional case.
From now on we assume that ${\widehat M}$ is odd dimensional.
From (\ref{E:3.5}) and (\ref{E:3.14}) we have

\begin{eqnarray}   \label{E:4.23}
\left( \Omega^{\bullet}(M, E)|_{Y} \right)_{\pm i} & = &  \left\{ \frac{1}{2} \left( I \mp \beta \Gamma^{Y} \right) \left( \begin{array}{clcr}  \omega_{1}     \\
\omega_{2} \end{array} \right) \hspace{0.1 cm} \mid \hspace{0.1 cm} \omega_{1}, \omega_{2} \in \Omega^{\bullet}(Y, E|_{Y}) \right\}.
\end{eqnarray}

\noindent
We recall the Calder\'on projector $\hspace{0.1 cm} {\mathcal C}_{1} : \Omega^{\bullet}(M, E)|_{Y} \rightarrow \Omega^{\bullet}(M, E)|_{Y} \hspace{0.1 cm}$ for $\B_{M_{1}}$ and the corresponding unitary operator
$\hspace{0.1 cm} K_{1} : \left( \Omega^{\bullet}(M, E)|_{Y} \right)_{+i} \rightarrow \left( \Omega^{\bullet}(M, E)|_{Y} \right)_{-i} \hspace{0.1 cm}$ so that $\graph(K_{1}) = \Imm {\mathcal C}_{1}$.
Let $\hspace{0.1 cm} {\widetilde {\mathcal C}}_{1} : \Omega^{\bullet}(M, E \oplus E)|_{Y} \rightarrow \Omega^{\bullet}(M, E \oplus E)|_{Y} \hspace{0.1 cm}$
be the Calder\'on projector for ${\widetilde \B}_{M_{1}}$.
From the definition of ${\widetilde \B}$ ((\ref{E:4.1})), we have

\begin{eqnarray}  \label{E:4.2323}
\Imm {\widetilde {\mathcal C}}_{1} & = & \Imm {\mathcal C}_{1} \oplus \Imm {\mathcal C}_{1}.
\end{eqnarray}

\noindent
Now consider the unitary operator
$\hspace{0.1 cm} {\widetilde K}_{1} : \left( \Omega^{\bullet}(M, E \oplus E)|_{Y} \right)_{+i} \rightarrow \left( \Omega^{\bullet}(M, E \oplus E)|_{Y} \right)_{-i} \hspace{0.1 cm}$ for
$\hspace{0.1 cm} {\widetilde {\mathcal C}}_{1}$.
Then, (\ref{E:4.2323}) implies that
for $x \in \left( \Omega^{\bullet}(M, E \oplus E)|_{Y} \right)_{+i}$, $\hspace{0.2 cm} x + {\widetilde K}_{1} x$ is expressed by
$\left( \begin{array}{clcr} y + K_{1}y \\ z + K_{1}z \end{array} \right)$ for some $y, z \in \left( \Omega^{\bullet}(M, E)|_{Y} \right)_{+i}$.
Hence, using (\ref{E:4.9}), ${\widetilde K}_{1}$ is described explicitly as follows.

\begin{eqnarray}   \label{E:4.24}
{\widetilde K}_{1} \left( \begin{array}{clcr} \omega_{1} \\ \omega_{2} \\ - \beta \Gamma^{Y} \omega_{1} \\ - \beta \Gamma^{Y} \omega_{2} \end{array} \right)
& = & {\widetilde K}_{1}
\left( \begin{array}{clcr} \frac{I - \beta \Gamma^{Y}}{2} \left( \begin{array}{clcr} \omega_{1} \\ \omega_{2} \end{array} \right) \hspace{0.1 cm} + \hspace{0.1 cm}
\frac{ I + \beta \Gamma^{Y} }{2} \left( \begin{array}{clcr} \omega_{1} \\ \omega_{2} \end{array} \right)  \\
\frac{ I - \beta \Gamma^{Y} }{2} \left( \begin{array}{clcr} \omega_{1} \\ \omega_{2} \end{array} \right) \hspace{0.1 cm} - \hspace{0.1 cm}
\frac{ I + \beta \Gamma^{Y} }{2} \left( \begin{array}{clcr} \omega_{1} \\ \omega_{2} \end{array} \right) \end{array} \right)  \nonumber  \\
& = &
\left( \begin{array}{clcr} K_{1} \frac{I - \beta \Gamma^{Y}}{2} \left( \begin{array}{clcr} \omega_{1} \\ \omega_{2} \end{array} \right) \hspace{0.1 cm} + \hspace{0.1 cm}
K_{1}^{-1} \frac{ I + \beta \Gamma^{Y} }{2} \left( \begin{array}{clcr} \omega_{1} \\ \omega_{2} \end{array} \right)  \\
K_{1} \frac{ I - \beta \Gamma^{Y} }{2} \left( \begin{array}{clcr} \omega_{1} \\ \omega_{2} \end{array} \right) \hspace{0.1 cm} - \hspace{0.1 cm}
K_{1}^{-1} \frac{ I + \beta \Gamma^{Y} }{2} \left( \begin{array}{clcr} \omega_{1} \\ \omega_{2} \end{array} \right) \end{array} \right).
\end{eqnarray}

\noindent
Since
$K_{1} \frac{I - \beta \Gamma^{Y}}{2} \left( \begin{array}{clcr} \omega_{1} \\ \omega_{2} \end{array} \right) \in \left( \Omega^{\bullet}(M, E)|_{Y} \right)_{-i}$ and
$K_{1}^{-1} \frac{ I + \beta \Gamma^{Y} }{2} \left( \begin{array}{clcr} \omega_{1} \\ \omega_{2} \end{array} \right) \in \left( \Omega^{\bullet}(M, E)|_{Y} \right)_{+i}$,
by (\ref{E:4.23}) we have

\begin{eqnarray*}
\beta \Gamma^{Y} K_{1} \frac{I - \beta \Gamma^{Y}}{2} \left( \begin{array}{clcr} \omega_{1} \\ \omega_{2} \end{array} \right) =
K_{1} \frac{I - \beta \Gamma^{Y}}{2} \left( \begin{array}{clcr} \omega_{1} \\ \omega_{2} \end{array} \right),
\hspace{0.2 cm}
\beta \Gamma^{Y} K_{1}^{-1} \frac{ I + \beta \Gamma^{Y} }{2} \left( \begin{array}{clcr} \omega_{1} \\ \omega_{2} \end{array} \right)  =
- K_{1}^{-1} \frac{ I + \beta \Gamma^{Y} }{2} \left( \begin{array}{clcr} \omega_{1} \\ \omega_{2} \end{array} \right),
\end{eqnarray*}

\noindent
which leads to

\begin{eqnarray}   \label{E:4.25}
& & U_{{\widetilde \PP}}^{-1} {\widetilde K}_{1} \left( \begin{array}{clcr} \omega_{1} \\ \omega_{2} \\ - \beta \Gamma^{Y} \omega_{1} \\ - \beta \Gamma^{Y} \omega_{2} \end{array} \right) \hspace{0.1 cm} = \hspace{0.1 cm}
\left( \begin{array}{clcr} S K_{1} \frac{I - \beta \Gamma^{Y}}{2} \left( \begin{array}{clcr} \omega_{1} \\ \omega_{2} \end{array} \right) \hspace{0.1 cm} + \hspace{0.1 cm}
S K_{1}^{-1} \frac{ I + \beta \Gamma^{Y} }{2} \left( \begin{array}{clcr} \omega_{1} \\ \omega_{2} \end{array} \right)  \\
- S K_{1} \frac{ I - \beta \Gamma^{Y} }{2} \left( \begin{array}{clcr} \omega_{1} \\ \omega_{2} \end{array} \right) \hspace{0.1 cm} + \hspace{0.1 cm}
S K_{1}^{-1} \frac{ I + \beta \Gamma^{Y} }{2} \left( \begin{array}{clcr} \omega_{1} \\ \omega_{2} \end{array} \right) \end{array} \right)  \nonumber  \\
& = &
\left( \begin{array}{clcr} S K_{1} \frac{I - \beta \Gamma^{Y}}{2} \left( \begin{array}{clcr} \omega_{1} \\ \omega_{2} \end{array} \right) \hspace{0.1 cm} + \hspace{0.1 cm}
S K_{1}^{-1} \frac{ I + \beta \Gamma^{Y} }{2} \left( \begin{array}{clcr} \omega_{1} \\ \omega_{2} \end{array} \right)  \\
- \beta \Gamma^{Y} \left\{ S K_{1} \frac{ I - \beta \Gamma^{Y} }{2} \left( \begin{array}{clcr} \omega_{1} \\ \omega_{2} \end{array} \right) \hspace{0.1 cm} + \hspace{0.1 cm}
S K_{1}^{-1} \frac{ I + \beta \Gamma^{Y} }{2} \left( \begin{array}{clcr} \omega_{1} \\ \omega_{2} \end{array} \right) \right\} \end{array} \right).
\end{eqnarray}

\vspace{0.2 cm}

\noindent
We define an isomorphism

\begin{eqnarray}   \label{E:4.26}
\Psi : \Omega^{\bullet}(M, E)|_{Y} \rightarrow \left( \Omega^{\bullet}(M, E \oplus E)|_{Y} \right)_{+i} \quad \text{by} \quad
\Psi \left( \begin{array}{clcr} \omega_{1} \\ \omega_{2} \end{array} \right) \hspace{0.1 cm} = \hspace{0.1 cm}
\left( \begin{array}{clcr} \omega_{1} \\ \omega_{2} \\ - \beta \Gamma^{Y} \omega_{1} \\ - \beta \Gamma^{Y} \omega_{2} \end{array} \right) ,
\end{eqnarray}

\noindent
which leads to

\begin{eqnarray}   \label{E:4.27}
\Psi^{-1} U_{{\widetilde \PP}}^{-1} {\widetilde K}_{1} \Psi \left( \begin{array}{clcr} \omega_{1} \\ \omega_{2} \end{array} \right) & = &
S K_{1} \frac{I - \beta \Gamma^{Y}}{2} \left( \begin{array}{clcr} \omega_{1} \\ \omega_{2} \end{array} \right) \hspace{0.1 cm} + \hspace{0.1 cm}
S K_{1}^{-1} \frac{ I + \beta \Gamma^{Y} }{2} \left( \begin{array}{clcr} \omega_{1} \\ \omega_{2} \end{array} \right) .
\end{eqnarray}

\noindent
By the same way, we have

\begin{eqnarray}   \label{E:4.28}
\Psi^{-1} {\widetilde K}_{1}^{-1} U_{{\widetilde \PP}} \Psi \left( \begin{array}{clcr} \omega_{1} \\ \omega_{2} \end{array} \right) & = &
K_{1} S \frac{I - \beta \Gamma^{Y}}{2} \left( \begin{array}{clcr} \omega_{1} \\ \omega_{2} \end{array} \right) \hspace{0.1 cm} + \hspace{0.1 cm}
K_{1}^{-1} S \frac{ I + \beta \Gamma^{Y} }{2} \left( \begin{array}{clcr} \omega_{1} \\ \omega_{2} \end{array} \right) .
\end{eqnarray}

\noindent
Since $\hspace{0.1 cm} \Omega^{\bullet}(M, E)|_{Y} =  \left( \Omega^{\bullet}(Y, E)|_{Y} \right)_{+i}  \oplus  \left( \Omega^{\bullet}(Y, E)|_{Y} \right)_{-i} \hspace{0.1 cm} $,
we use this decomposition to write

\begin{eqnarray}   \label{E:4.29}
\Psi^{-1} \left( U_{{\widetilde \PP}}^{-1} {\widetilde K}_{1} + {\widetilde K}_{1}^{-1} U_{{\widetilde \PP}} \right) \Psi & = &
 \left( \begin{array}{clcr}  0 & S K_{1}^{-1} + K_{1}^{-1} S  \\ S K_{1} + K_{1} S  &  0  \end{array} \right) \nonumber   \\
& = & \left( \begin{array}{clcr}  0 & ( S K_{1} + K_{1} S)^{\ast}  \\ S K_{1} + K_{1} S  &   0  \end{array} \right) .
\end{eqnarray}

\noindent
Hence, we have

\begin{eqnarray}   \label{E:4.30}
\ddet^{\ast}_{\Fr} \left( I + \frac{1}{2} \left( U_{{\widetilde \PP}}^{-1} {\widetilde K}_{1} + {\widetilde K}_{1}^{-1} U_{{\widetilde \PP}} \right) \right) & = &
\ddet^{\ast}_{\Fr} \left( \begin{array}{clcr}  I & \frac{1}{2} ( S K_{1} + K_{1} S)^{\ast}  \\ \frac{1}{2} (S K_{1} + K_{1} S)  & I  \end{array} \right)  \nonumber \\
& = & \ddet^{\ast}_{\Fr} \left( I - \frac{1}{4} ( S K_{1} + K_{1} S)^{\ast} (S K_{1} + K_{1} S) \right).
\end{eqnarray}

\noindent
To analyze (\ref{E:4.30}), we note that

\begin{eqnarray}   \label{E:4.99}
\left(\Omega^{\bullet}(Y, E)|_{Y}\right)_{\pm i} & = &
\left( \begin{array}{clcr} \Omega^{\bullet}(Y, E|_{Y})_{\pm} \\ \oplus \\  \Omega^{\bullet}(Y, E|_{Y})_{\pm} \end{array} \right), \quad
\text{where} \quad
\Omega^{\bullet}(Y, E|_{Y})_{\pm} := \frac{I \mp \beta \Gamma^{Y}}{2} \Omega^{\bullet}(Y, E|_{Y}).
\end{eqnarray}

\noindent
According to this decomposition, we may write
$K_{1} : \left( \begin{array}{clcr} \Omega^{\bullet}(Y, E|_{Y})_{+} \\ \oplus \\  \Omega^{\bullet}(Y, E|_{Y})_{+} \end{array} \right) \rightarrow
\left( \begin{array}{clcr} \Omega^{\bullet}(Y, E|_{Y})_{-} \\ \oplus \\  \Omega^{\bullet}(Y, E|_{Y})_{-} \end{array} \right)$
by

\begin{eqnarray}   \label{E:4.31}
K_{1} & = & \left( \begin{array}{clcr} A_{1} & B_{1} \\ C_{1} & D_{1} \end{array} \right), \quad \text{where} \quad
A_{1}, B_{1}, C_{1}, D_{1} : \Omega^{\bullet}(Y, E|_{Y})_{+} \rightarrow \Omega^{\bullet}(Y, E|_{Y})_{-}.
\end{eqnarray}

\noindent
We note that $\Gamma = i \beta \Gamma^{Y} \left( \begin{array}{clcr} 0 & -1 \\ 1 & 0 \end{array} \right)$ preserves the decomposition
$\left( \Omega^{\bullet}(Y, E)|_{Y} \right)_{\pm i}$ and commutes with $\B_{M_{1}}$, which implies that
$K_{1}$ commutes with $\Gamma$.
Since $\Omega^{\bullet}(Y, E|_{Y})_{\pm} \hspace{0.1 cm}$ are $(\mp 1)$-eigenspaces of $\beta \Gamma^{Y}$,
we have
$\left( \begin{array}{clcr} 0 & -1 \\ 1 & 0 \end{array} \right) K_{1} = K_{1} \left( \begin{array}{clcr} 0 & 1 \\ -1 & 0 \end{array} \right)$,
which shows that

\begin{eqnarray}   \label{E:4.33}
B_{1} \hspace{0.1 cm} = \hspace{0.1 cm} C_{1}, \qquad A_{1} \hspace{0.1 cm} = \hspace{0.1 cm} - D_{1} .
\end{eqnarray}

\noindent
Hence, we have

\begin{eqnarray}   \label{E:4.34}
S K_{1} + K_{1} S & = & \left( \begin{array}{clcr} 2 A_{1} & 0 \\ 0 & 2 A_{1} \end{array} \right) .
\end{eqnarray}

\noindent
Since $\hspace{0.1 cm} K_{1} - U_{\Pi_{>}} \hspace{0.1 cm}$ is a trace class operator ((\ref{E:2.10})) and
$\hspace{0.1 cm} U_{\Pi_{>}} = (\B_{Y}^{2})^{-1} (\nabla^{Y} + \Gamma^{Y} \nabla^{Y} \Gamma^{Y}) \left( \begin{array}{clcr} 0 & -1  \\ -1 & 0 \end{array} \right)$
 ((\ref{E:3.15})),  $A_{1}$ is a trace class operator.
Hence, we have

\begin{eqnarray}   \label{E:4.35}
& & \ddet^{\ast}_{\Fr} \left( I + \frac{1}{2} \left( U_{{\widetilde \PP}}^{-1} {\widetilde K}_{1} + {\widetilde K}_{1}^{-1} U_{{\widetilde \PP}} \right) \right) \hspace{0.1 cm} = \hspace{0.1 cm}
\ddet^{\ast}_{\Fr} \left( I - \frac{1}{4} ( S K_{1} + K_{1} S)^{\ast} (S K_{1} + K_{1} S) \right) \nonumber  \\
& = & \ddet_{\Fr}^{\ast} \left( \begin{array}{clcr} I - A_{1}^{\ast} A_{1} & 0 \\ 0 & I - A_{1}^{\ast} A_{1} \end{array} \right)
\hspace{0.1 cm} = \hspace{0.1 cm}  \left( \ddet_{\Fr}^{\ast} \left( I - A_{1}^{\ast} A_{1} \right) \right)^{2}.
\end{eqnarray}

\noindent
Putting $K_{2}  =  \left( \begin{array}{clcr} A_{2} & B_{2} \\ C_{2} & D_{2} \end{array} \right)$,
the same method shows that

\begin{eqnarray}   \label{E:4.37}
\ddet^{\ast}_{\Fr} \left( I - \frac{1}{2} \left( U_{{\widetilde \PP}}^{-1} {\widetilde K}_{2} + {\widetilde K}_{2}^{-1} U_{{\widetilde \PP}} \right) \right) & = &
 \left( \ddet_{\Fr}^{\ast} \left( I - A_{2}^{\ast} A_{2} \right) \right)^{2}.
\end{eqnarray}

\noindent
In view of (\ref{E:4.21}) we note that

\begin{eqnarray}   \label{E:4.38}
{\widetilde K}_{1}^{-1} {\widetilde K}_{2} \left( \begin{array}{clcr} \omega_{1} \\ \omega_{2} \\ - \beta \Gamma^{Y} \omega_{1} \\ - \beta \Gamma^{Y} \omega_{2} \end{array} \right) & = &
\left( \begin{array}{clcr} K_{1}^{-1} K_{2} \frac{I - \beta \Gamma^{Y}}{2} \left( \begin{array}{clcr} \omega_{1} \\ \omega_{2} \end{array} \right) \hspace{0.1 cm} + \hspace{0.1 cm}
K_{1} K_{2}^{-1} \frac{ I + \beta \Gamma^{Y} }{2} \left( \begin{array}{clcr} \omega_{1} \\ \omega_{2} \end{array} \right)  \\
K_{1}^{-1} K_{2} \frac{ I - \beta \Gamma^{Y} }{2} \left( \begin{array}{clcr} \omega_{1} \\ \omega_{2} \end{array} \right) \hspace{0.1 cm} - \hspace{0.1 cm}
K_{1} K_{2}^{-1} \frac{ I + \beta \Gamma^{Y} }{2} \left( \begin{array}{clcr} \omega_{1} \\ \omega_{2} \end{array} \right) \end{array} \right)    \\
& = &
\left( \begin{array}{clcr} K_{1}^{-1} K_{2} \frac{I - \beta \Gamma^{Y}}{2} \left( \begin{array}{clcr} \omega_{1} \\ \omega_{2} \end{array} \right) \hspace{0.1 cm} + \hspace{0.1 cm}
K_{1} K_{2}^{-1} \frac{ I + \beta \Gamma^{Y} }{2} \left( \begin{array}{clcr} \omega_{1} \\ \omega_{2} \end{array} \right)  \\
- \beta \Gamma^{Y} \left\{ K_{1}^{-1} K_{2} \frac{ I - \beta \Gamma^{Y} }{2} \left( \begin{array}{clcr} \omega_{1} \\ \omega_{2} \end{array} \right) \hspace{0.1 cm} + \hspace{0.1 cm}
K_{1} K_{2}^{-1} \frac{ I + \beta \Gamma^{Y} }{2} \left( \begin{array}{clcr} \omega_{1} \\ \omega_{2} \end{array} \right)  \right\} \end{array} \right) ,  \nonumber
\end{eqnarray}

\noindent
which shows that

\begin{eqnarray}   \label{E:4.39}
\Psi^{-1} {\widetilde K}_{1}^{-1} {\widetilde K}_{2} \Psi  & = &
 K_{1}^{-1} K_{2} \frac{I - \beta \Gamma^{Y}}{2}  \hspace{0.1 cm} + \hspace{0.1 cm}
K_{1} K_{2}^{-1} \frac{ I + \beta \Gamma^{Y} }{2} \hspace{0.1 cm} = \hspace{0.1 cm}
\left( \begin{array}{clcr} K_{1}^{-1} K_{2} & 0 \\ 0 & K_{1} K_{2}^{-1} \end{array} \right).
\end{eqnarray}

\noindent
Hence, we have

\begin{eqnarray}   \label{E:4.40}
\ddet_{\Fr}^{\ast} \left( \frac{1}{2} \left( I - {\widetilde K}_{1}^{-1} {\widetilde K}_{2} \right) \right)  & = &
\mid \ddet_{\Fr}^{\ast} \left( \frac{1}{2} \left( I - K_{1}^{-1} K_{2} \right) \right)  \mid^{2} .
\end{eqnarray}

The same computation for an even dimensional case leads to the same result.
Summarizing the above argument, we have the following result, which is the main result of this section.

\vspace{0.2 cm}

\begin{theorem}   \label{Theorem:4.2}
Let $({\widehat M}, {\widehat g})$ be a closed Riemannian manifold and $Y$ be a hypersurface of ${\widehat M}$ with ${\widehat M} = M_{1} \cup_{Y} M_{2}$.
We assume that ${\widehat g}$ is a product metric near $Y$.
We denote the odd signature operator on ${\widehat M}$ by $\B_{{\widehat M}}$ and its restriction to $M_{1}$ and $M_{2}$ by
$\B_{M_{1}}$ and $\B_{M_{2}}$.  Then :

\begin{eqnarray*}
&  & \sum_{q=0}^{m} \left( \log \Det^{\ast} \B^{2}_{{\widehat M}, q} - \log \Det^{\ast} \B^{2}_{M_{1}, q, \rel} - \log \Det^{\ast} \B^{2}_{M_{2}, q, \Abs} \right)  \hspace{0.1 cm}  =  \hspace{0.1 cm}  - \log 2 \sum_{q=0}^{m} \left( \beta_{q}(M_{1}) + \beta_{q}(M_{2}) \right)   \nonumber  \\
& + &    2 \sum_{q=0}^{m} \log \ddet A_{0, q} \hspace{0.1 cm} - \hspace{0.1 cm} 2
\sum_{q=0}^{m} \left( \log \ddet V_{M_{1, q, \rel}} + \log \ddet V_{M_{2, q, \Abs}} \right)
\hspace{0.1 cm}  + \hspace{0.1 cm} 2 \log | \ddet^{\ast}_{\Fr} \left( \frac{1}{2} \left( I - K_{1}^{-1} K_{2} \right) \right) |   \nonumber \\
& - &  \left\{ \log \ddet^{\ast}_{\Fr} \left( I - A_{1}^{\ast} A_{1} \right) +
\log \ddet^{\ast}_{\Fr} \left( I - A_{2}^{\ast} A_{2} \right) \right\},
\end{eqnarray*}

\noindent
where $A_{1}, A_{2} : \Omega^{\bullet}(Y, E|_{Y})_{+} \rightarrow \Omega^{\bullet}(Y, E|_{Y})_{-}$ are first components of $K_{1}$ and $K_{2}$, respectively.
\end{theorem}

\vspace{0.2 cm}

\noindent
{\it Remark} : (1) If all cohomologies vanish, {\it i.e.} $H^{\ast}(M ; E) = H^{\ast}(M_{i} ; E) = H^{\ast}(M_{i}, Y ; E) = 0$, then the first three terms in Theorem \ref{Theorem:1.1}
and Theorem \ref{Theorem:4.2} do not appear.   \newline
(2) So far we don't know how to describe
$\left( \log \Det^{\ast} \B^{2}_{{\widehat M}, q} - \log \Det^{\ast} \B^{2}_{M_{1}, q, {\mathcal P}_{-, {\mathcal L}_{0}}} -
\log \Det^{\ast} \B^{2}_{M_{2}, q, {\mathcal P}_{+, {\mathcal L}_{1}}} \right)$ and
$\left( \log \Det^{\ast} \B^{2}_{{\widehat M}, q} - \log \Det^{\ast} \B^{2}_{M_{1}, q, \rel} - \log \Det^{\ast} \B^{2}_{M_{2}, q, \Abs} \right)$
for each single $q$.


\end{document}